\documentclass{article}

\usepackage{PRIMEarxiv}

\usepackage[utf8]{inputenc} 
\usepackage[T1]{fontenc}    
\usepackage{hyperref}       
\usepackage{url}            
\usepackage{booktabs}       
\usepackage{amsfonts}       
\usepackage{nicefrac}       
\usepackage{microtype}      
\usepackage{lipsum}
\usepackage{fancyhdr}       
\usepackage{graphicx}       
\usepackage{amsmath}
\usepackage{amsthm}
\usepackage{appendix}
\usepackage{cleveref}
\usepackage{subcaption}
\graphicspath{{media/}}     

\theoremstyle{plain} 
\newtheorem{lemma}{Lemma}[section]
\newtheorem{theorem}{Theorem}[section]
\newtheorem{proposition}{Proposition}[section]
\newtheorem{corollary}{Corollary}[section]

\theoremstyle{remark}
\newtheorem{remark}{Remark}

\theoremstyle{definition} 
\newtheorem{definition}{Definition}[section]
\newtheorem{example}{Example}[section]

\title{On the convergence of fictitious play algorithm in repeated games via the geometrical approach
}

\author{
  Zhouming Wu \\
  Key Lab of Systems and Control, ISS \\
  Academy of Mathematics and Systems Science \\
  CAS Beijing 10094, China\\
  \texttt{wuzhouming@amss.ac.cn} \\
   \And
  Yifen Mu \\
  Key Lab of Systems and Control, ISS \\
  Academy of Mathematics and Systems Science \\
  CAS Beijing 10094, China\\
  \texttt{mu@amss.ac.cn} \\
  \AND
  Xiaoguang Yang \\
  Key Lab of Management, Decision and Information System, ISS \\
  Academy of Mathematics and Systems Science \\
  CAS Beijing 10094, China\\
  \texttt{xgyang@iss.ac.cn} \\
}

\begin{document}
\maketitle

\begin{abstract}
As the earliest and one of the most fundamental learning dynamics for computing NE, fictitious play (FP) has being receiving incessant research attention and finding games where FP would converge (games with FPP) is one central question in related fields. In this paper, we identify a new class of games with FPP, i.e., $3\times3$ games without IIP, based on the geometrical approach by leveraging the location of NE and the partition of best response region. 
During the process, we devise a new projection mapping to reduce a high-dimensional dynamical system to a planar system. And to overcome the non-smoothness of the systems, we redefine the concepts of saddle and sink NE, which are proven to exist and help prove the convergence of CFP by separating the projected space into two parts. Furthermore, we show that our projection mapping can be extended to higher-dimensional and degenerate games.
\end{abstract}


\section{Introduction}
Learning algorithms provide iterative methods to solve Nash equilibrium (NE), which was proposed by \cite{nash1950non} in 1950 and is the most widely accepted solution concept in game theory. Among all the learning algorithms, fictitious play (FP) is the earliest and fundamental one \cite{brown1951iterative, fudenberg1998theory, krishna1998convergence}. And it still attracts research attention now, with numerous applications \cite{heinrich2015fictitious, heinrich2016deep} and extensions into more complex settings \cite{sayin2022fictitious0-sum, sayin2022fictitious}.

The idea behind FP is intuitive: each player estimates his opponent's mixed strategy to be the empirical distribution of the pure strategies during the history, and then takes a best response to the estimation. 
The game is said to have fictitious play property (FPP), if the empirical distribution of FP converges to NE. 
It have been proved that FP can converge to NE in zero-sum games \cite{robinson1951iterative}, $2\times 2$ games \cite{miyasawa2010convergence}, potential games \cite{monderer1996fictitious}, $ 2\times n $ games \cite{bergerFictitiousPlayGames2005} and other games with special properties, such as games solvable by iterated strict dominance \cite{nachbar1990evolutionary}, non-degenerate quasi-supermodular games with diminishing returns or of dimension $3\times 
m$ or $4\times 4$ \cite{bergerTwoMoreClasses2007}. Recently, based on decomposition of games
, it has been proven that if a game is zero-sum or potential, then FP would also converge to NE in the linear combination of some decomposed components of the original game.

However, there are games in which FP does not converge. As early as 1964, \cite{shapley1963some} constructed a simple $3\times3$ counterexample. Other classes of games without FPP includes \cite{hahnShapleyPolygonsGames2010, jordan1993three}.  Recently, Shapley game was extended to a family of games with one parameter, and the authors found that under FP dynamics, bifurcation would emerge and more complicated pattern including erratic behavior would appear as the parameter changes \cite{sparrow2008fictitious,van2011fictitious}. 
In fact, this non-convergence happens for general learning algorithms. For example, the algorithm "Follow the Regularized Leader"(FTRL) is proved to be Poincar\'e recurrent in zero-sum games \cite{mertikopoulos2018cycles}. For $2\times2$ competitive games, \cite{muthukumar2024impossibility} proved that no-regret algorithms which are mean-based and monotonic cannot converge to any NE. And researchers even proved that learning in matrix games can approximate arbitrary dynamical systems including Lorenz system, hence surely does not converge \cite{andrade2021learning}. The latest research \cite{milionisImpossibilityTheoremGame2023} shows that there exists some game that no learning algorithm could converge to NE.  

Therefore concerning FP, a central question is to identify the classes of games having FPP, and the problem is attracting the researchers' attention incessantly \cite{bergerTwoMoreClasses2007,chen2022convergence}. In this paper, we will study this problem, and identify a novel game class with FPP, which is based on geometrical properties and analysis of the dynamics in the strategy simplex, combining the theory of dynamical systems and other geometrical tools. Since two games with completely different payoff matrices could have the same geometrical property and hence induce the same FP dynamics, it is more natural to prove the convergence by geometrical approach.

Geometrical approach was ever utilized to analyze Replicator Dynamics (RD) in quite a lot of literature, since RD is formulated by an ordinary differential equation and hence induces a dynamical system naturally, see the chapter 3 of \cite{fudenberg1998theory} for summation. Concerning the analysis of FP, the approach can be at least traced back to \cite{metrick1994fictitious}, which gives a new geometrical proof of the convergence of fictitious play in $2\times 2$ games. Recently, such geometrical approach to study FP dynamics is receiving more and more attention \cite{bergerFictitiousPlayGames2005,sparrow2008fictitious, van2011fictitious, hahnShapleyPolygonsGames2010, van2011hamiltonian,berger2012non}. In these works, researchers mainly focus on the continuous-time version of FP (CFP), which serves as a scaffolding to understand discrete-time cases because of its analytical simplicity. For example, In zero-sum games, CFP can inspire Hamiltonian dynamics which was used to prove the convergence results by \cite{van2011hamiltonian}. And for games with some specific geometrical properties, CFP trajectory can be analyzed by the theory of projective geometry, which helps to prove the convergence \cite{berger2012non}. Besides, using geometrical approach,  studied a family of $4\times 4$ games and demonstrated the conditions as well as examples under which CFP does not converge \cite{hahnShapleyPolygonsGames2010}. Since the analysis is implemented on the strategy simplex, the geometrical approach is more straightforward and can greatly shorten the exposition \cite{metrick1994fictitious}, compared with the algebraic method that uses the property about the payoff matrix elements, such as zero-sum, potential, (quasi-)supermodular or diminishing returns. 

To use geometrical approach, two main challenges need to be overcome. The first challenge is the dimension of phrase space, which would increases rapidly as the payoff matrices become large. Even for $3\times 3$ games, we will encounter a four-dimensional dynamical system. To deal with it, we propose a new projection method to reduce dimension of $ 3\times 3 $ games and transfer the CFP dynamics from four-dimensional into planar system, on which the theory is more developed. And hence we could further define a Poincar\'e map as a tool to analyze the convergence. 

To define the projection mapping and identify the game classes, we highlight a novel geometrical object of the game, the indifferent point, which is a generalization of strategy profile and is proved to exist uniquely for any $n\times n$ games. At the indifferent point, for each player, all actions have the same "payoff", while different from ordinary strategy profile, the probability of some pure strategy could be negative but the sum of all probabilities is still equal to 1. Geometrically, the indifferent point is possible to be inside (e.g. Shapley game) or outside the strategy simplex. If the indifferent point lies outside the strategy simplex, we will call the game without \emph{Internal Indiffernt Point (IIP)}, which is the focus of this paper. For such games without IIP, we can define the projection mapping and prove the convergence of CFP dynamics on the projected space.

The second challenge is related to the nature of FP or CFP: the vector field induced by CFP is not continuous. Hence, we have to study piecewise vector fields \cite{euzebio2023periodic},  which makes it impossible to analyze the stability of NE by calculating the Jacobian matrix. Moreover, common concepts such as saddle, sink cannot be directly applied. To deal with this, we redefine the saddle and sink for $3\times 3$ games according to the geometrical location of NE, and prove the existence of saddle NE and sink NE for games with multiple NEs. Under the new definition, the behavior of CFP dynamics near the saddle or sink is similar to smooth system. And we will show that the existence of saddle NE plays an important role to prove the convergence. 

Briefly, this paper makes three main contributions. First, based only on the geometrical property of games and learning dynamics, we identify a novel class of games without IPP, where FP would converge to NE. Second, for this class of games, we prove some properties, including the general existence and uniqueness of the indifferent point as well as the existence of saddle NE and sink NE. Third, we define a new projection method to study game dynamics, which can be extended more game classes, including degenerate games and $4\times 4$ games. Moreover, the geometrical approach enable us to detect many interesting phenomena, e.g. for a family of games with only one element of payoff matrix changing, we find that not only CFP dynamics but also the NE show the abrupt regime shift. 

This paper differs from the results or method in the literature in several aspects. 
Compared to the recently reported game classes with FPP, i.e., $3\times m$ and $4\times 4$ quasi-supermodular games \cite{bergerTwoMoreClasses2007}, this paper introduce a new class, i.e. $3\times3$ games without IPP. By constructing an example in Section \ref{sec4.1}, we can show that there exists some game within our class but not  quasi-supermodular. As for the approach, this paper is inspired to reduce the dimension by \cite{bergerFictitiousPlayGames2005, sparrow2008fictitious, van2011fictitious}, but the specific method is quite different. In \cite{bergerFictitiousPlayGames2005}, the author considered $2\times n$ games and then used the payoff matrix to project one player's $n$-dimensional strategy vector to a $2$-dimensional payoff vector. Different from it, this paper projects higher-dimensional strategy to lower-dimensional strategy directly. In \cite{sparrow2008fictitious, van2011fictitious}, the authors also mentioned the projection method, but since they mainly focused on $3\times3$ games with fully-mixed NE, the projection mapping in this paper cannot be well defined.

This paper is organized as follows. Section \ref{sec2:preliminary} gives the problem formulation and some necessary background about the non-convergent dynamics. Section \ref{sec3:result} presents our projection method, and we utilize it to prove the convergence of CFP by two cases: games with unique NE or multiple NEs. Section \ref{sec4:further} first compares our result with \cite{bergerTwoMoreClasses2007} by constructing an example, and further extends the projection method to degenerate games and $4\times4$ games with interesting phenomena on CFP dynamics. Section \ref{sec5:conclusion} summarizes the paper and gives future work.

\section{Preliminary}\label{sec2:preliminary}
\subsection{Games and Nash Equilibrium}
Let $(A,B)$ be a two-player normal form game where $A,B$ are both $3\times 3$ payoff matrices corresponding to Player A and Player B respectively. Player A has pure strategies ${e^{A}_{1},e^{A}_{2},e^{A}_{3}}$ and Player B has pure strategies $e^{B}_{1},e^{B}_{2},e^{B}_{3} $. A mixed strategy is a distribution over all pure strategies of each player, and all mixed strategies form a strategy simplex $\Delta_{A}:=\{\mathbf{x}\in\mathbb{R}^{3}\mid x_{i}\geq0,\sum_{i=1}^{3}x_{i}=1\}$ for Player A, and $\Delta_{B}$ for Player B. Hence $e^{A}_{i}$ is the vertex of $\Delta_{A}$ whose $i$-th element is $1$ and other are $0$. And we use the endpoints to denote the edge of $ \Delta_{A} $, for example, the edge $ \{\mathbf{x}\in\Delta_{A}\mid x_{3}=0,x_{1}\leq 0, x_{2}\leq0, x_{1}+x_{2}=1\} $ is denoted as $ e^{A}_{1}e^{A}_{2} $. 

Let $\Delta:=\Delta_{A}\times\Delta_{B}$. Given a mixed strategy profile $(\mathbf{x},\mathbf{y})\in\Delta$, the expected utility of Player A (or B) is $\mathbf{x}^{T}A\mathbf{y}$ (or $\mathbf{x}^{T}B\mathbf{y}$, respectively). Hence when Player A uses pure strategy $e^{A}_{i}$ against $y$, his utility is $(Ay)_{i}$. We define 
\[ \operatorname{BR}_{A}(\mathbf{y}):=\operatorname{\arg\max}_{e^{A}_{i}\in E^{A}}(Ay)_{i}, \operatorname{BR}_{B}(\mathbf{x}):=\operatorname{\arg\max}_{e^{B}_{j}\in E^{B}}(x^{T}B)_{j}, \]
which is a set usually.

Nash's well-known concept of equilibrium is as below.

\begin{definition}\label{def:NE}
	A mixed strategy profile $ (\bar{\mathbf{x}},\bar{\mathbf{y}}) $ is called a Nash equilibrium of $ (A,B) $, if
	$$\begin{aligned}
		&\bar{\mathbf{x}}^{T}A\bar{\mathbf{y}}\geq (\mathbf{x}^{\prime})^{T}A\bar{\mathbf{y}}, &\forall \mathbf{x}^{\prime}\in \Delta_{A}, \mathbf{x}^{\prime} \neq \bar{\mathbf{x}};\\ &\bar{\mathbf{x}}^{T}B\bar{\mathbf{y}}\geq \bar{\mathbf{x}}^{T}B\mathbf{y}^{\prime}, &\forall\mathbf{y}^{\prime}\in\Delta_{B}, \mathbf{y}^{\prime} \neq \bar{\mathbf{y}}.
	\end{aligned}
	$$ 
\end{definition}

For some player's mixed strategy $ \mathbf{x} $, its support is defined as the set of all pure actions with positive probability, i.e.,
$$
\operatorname{supp}(\mathbf{x})=\{e_{i}\mid x_{i}>0, \text{ for }\mathbf{x}=(x_{1},x_{2},x_{3})\}.
$$ 
We say a strategy profile $ (\mathbf{x},\mathbf{y}) $ has support with size $ m\times n $, if $ \vert\operatorname{supp}(\mathbf{x})\vert =m $, and $\vert\operatorname{supp}(\mathbf{y})\vert = n $. Hence a pure NE has support with size $ 1\times 1 $, and a fully mixed NE has support with size $ 3\times 3 $. Based on this concept, we have the following equalization principle which characterizes the relationship between NE and its support. 

\begin{lemma}[Equalization Principle (\cite{gtalive}, Theorem 7.1 \cite{bookNarahari})] 
	\label{lem:indifferent-of-support}
	Given a normal form game $ (A,B) $, the mixed strategy profile $(\mathbf{x},\mathbf{y})$ is a NE iff for any player $ k=A,B $, given the opponent's mixed strategy $\mathbf{y}$, his own strategy $\mathbf{x}$ satisfies
	\begin{enumerate}
		\item $ u_{k}(e_{j}^{k},\mathbf{y}) $ is the same for all $ e_{j}^{k}\in \operatorname{supp}(\mathbf{x}) $, and
		\item $ u_{k}(e_{j}^{k},\mathbf{y})\geq  u_{k}(e_{j^{\prime}}^{k},\mathbf{y})$, where $e_{j}^{k}\in\operatorname{supp}(\mathbf{x})$ and $e_{j^{\prime}}^{k}\notin \operatorname{supp}(\mathbf{x}) $.
	\end{enumerate}
\end{lemma}

Usually the NE of a game is not unique. One surprising but beautiful result is that for nondegenerate games, NEs are isolated, and the number is always finite and odd (\cite{wilsonComputingEquilibriaNperson1971,harsanyiOddnessNumberEquilibrium1973}). On the contrary, for degenerate games, which takes zero measurement in payoff matrix space, NE can constitute a continuum. This paper focuses mainly on the game with finite and isolated NEs, but we will also discuss degenerate games at the end part, which shows interesting phenomena.

\begin{remark}\label{rmk:degenerate}
	There is no generally accepted definition of degenerate games. For example, some literature defines a game to be nondegenerate if for any pure strategy of each player there is a unique best response \cite{bergerFictitiousPlayGames2005}, which is also adopted in this paper. Some other literature just implicitly separates those two cases according to whether the number of equilibrium is finite \cite{milionisImpossibilityTheoremGame2023}.
\end{remark}

\subsection{Fictitious Play} 
In the fictitious play setting, players play the game $(A,B)$ repeatedly. And they treat the opponent's empirical distribution until time $t$ as the estimation of opponent's mixed strategy at time $t+1$, then choose the best response against it. Denote pure strategy of Player A at time $t$ by $ \mathbf{a}({t})\in\{e^{A}_{1},e^{A}_{2},e^{A}_{3}\} $, and pure strategy of Player B as $ \mathbf{b}({t})$. Then the empirical distribution $\mathbf{x}(t)$ of Player A and $\mathbf{y}(t)$ of Player B are
\[ \mathbf{x}(t)=\frac{\sum_{\tau=1}^{t}\mathbf{a}(\tau)}{t},\quad \mathbf{y}(t)=\frac{\sum_{\tau=1}^{t}\mathbf{b}(\tau)}{t}. \]
Given $ (\mathbf{x}(t),\mathbf{y}(t)) $, each player will take the best response, i.e.,  \begin{equation*}
	\mathbf{a}({t+1})=\operatorname{BR}_{A}\left(\mathbf{y}(t)\right),\ \mathbf{b}({t+1})=\operatorname{BR}_{B}\left(\mathbf{x}(t)\right). 
\end{equation*}
Given the initial point $(\mathbf{x}(0), \mathbf{y}(0))\in\Delta$, the dynamics can be represented by the following recursive equations 
\begin{equation} \label{equ:DFP}
	\left\{
	\begin{array}{lr}
		\mathbf{x}(t+1) = \frac{1}{t+1}\Big(\operatorname{BR}_{A}\left(\mathbf{y}(t)\right)+t\cdot \mathbf{x}(t)\Big),\\
		\mathbf{y}(t+1) = \frac{1}{t+1}\Big(\operatorname{BR}_{B}\left(\mathbf{x}(t)\right)+t\cdot \mathbf{y}(t)\Big).
	\end{array}
	\right.
\end{equation}
Equation (\ref{equ:DFP}) is the updating rule of Discrete-time Fictitious Play (DFP, or FP for short). Since the best response is not always a singleton, some literature would specify a tie-breaking rule in prior. For convenience, we suppose both players would always choose the pure strategy with lower index when there are multiple choices. 

If we go from discrete-time step to continuous-time, we can get Continuous-time Fictitious Play (CFP), with $t\in[1,\infty)$,
\begin{equation} 
	\label{equ:cfp-def}
	\begin{cases}
		\dot{x} = \frac{1}{t}\big(\operatorname{BR}_{A}(y(t))-x(t)\big),\\
		\dot{y} = \frac{1}{t}\big(\operatorname{BR}_{B}(x(t))-y(t)\big).
	\end{cases}
\end{equation}
CFP can provide clearer understanding to the dynamics and give much more convenience for analysis. 

It is worth noting that CFP is similar to the Best Response Dynamic (BRD), whose equations are
\begin{equation} 
	\label{equ:cbr-def}
	\left\{
	\begin{array}{lr}
		\dot{x} = \operatorname{BR}_{A}(y(t))-x(t),\\
		\dot{y} = \operatorname{BR}_{B}(x(t))-y(t).
	\end{array}
	\right.
\end{equation}
The only difference between dynamics \ref{equ:cfp-def} and (\ref{equ:cbr-def}) is that in (\ref{equ:cbr-def}), $ \frac{1}{t} $ is omitted, leading to an exponentially slower evolution in CFP, while they would share 
the same orbit shape. See detailed analysis in the next section. 

Given the dynamical system (\ref{equ:DFP})  and initial point $ (\mathbf{x}(0),\mathbf{y}(0)) $, FP is said to \emph{converge}, if there exists some strategy profile $(\mathbf{x}^{*},\mathbf{y}^{*})\in\Delta$ such that $ \big(\mathbf{x}(t),\mathbf{y}(t)\big)\to (\mathbf{x}^{*},\mathbf{y}^{*}) $ as $ t\to\infty $. In the literature, it is a central question to distinguish the classes of games in which FP would converge. One of the milestone results is
 
\begin{lemma}[\cite{bergerFictitiousPlayGames2005}]\label{lem:2xn}
	Both FP and CFP converges to Nash equilibrium when one of two players has only $2$ pure strategies.
\end{lemma}

By Lemma \ref{lem:2xn}, when one of two players uses only $2$ pure strategies, FP would converge, which we would utilize later.

\subsection{Nonconvergent result}
People may wish FP converge to NE in more games. However Shapley's counterexample shatters this hope by giving a simple $ 3\times 3 $ game as 
\begin{equation}\label{equ:Shpley-payoff}
	A=\begin{bmatrix}
		1     & 0     & 0 \\[1.5ex]
		0     & 1     & 0     \\[1.5ex]
		0     & 0     & 1
	\end{bmatrix},\quad 
	B = \begin{bmatrix}
		0      & 1      & 0      \\[1.5ex]
		0      & 0      & 1      \\[1.5ex]
		1      & 0      & 0
	\end{bmatrix}.
\end{equation}
The unique equilibrium of game (\ref{equ:Shpley-payoff}) is $ x=(\frac{1}{3},\frac{1}{3},\frac{1}{3})^{T},y=(\frac{1}{3},\frac{1}{3},\frac{1}{3})^{T} $. Then Shapley proved that FP would not converge for this game. Recently, it was found that FP would converge to a limit cycle \cite{sparrow2008fictitious,van2011fictitious}. 

For more general $ 3\times 3 $ games, the trajectory of FP would be much more complicated. \cite{ostrovski2014payoff} gives an example, whose payoff matrices are
\begin{equation}\label{equ:8-game}
	A = \begin{bmatrix}
		-1.35 & -1.26 & 2.57\\[1.5ex]
		0.16 & -1.80 & 1.58\\[1.5ex]
		-0.49 & -1.54 & 1.9
	\end{bmatrix}, B = \begin{bmatrix}
		-1.83 & -2.87 & -3.36\\[1.5ex]
		-4.80 & -3.85 & -3.75\\[1.5ex]
		6.740 & 6.59 & 6.89
	\end{bmatrix}.
\end{equation}
And its unique NE is $ x=(0.288, 0.370, 0.342)^{T} $,
$ y=(0.335, 0.327, 0.338)^{T} $. Figure \ref{fig:shapley-8} shows the dynamics in games (\ref{equ:Shpley-payoff}) and (\ref{equ:8-game}). 

\begin{figure}[htbp]
	\centering
	\includegraphics[width=0.8\linewidth]{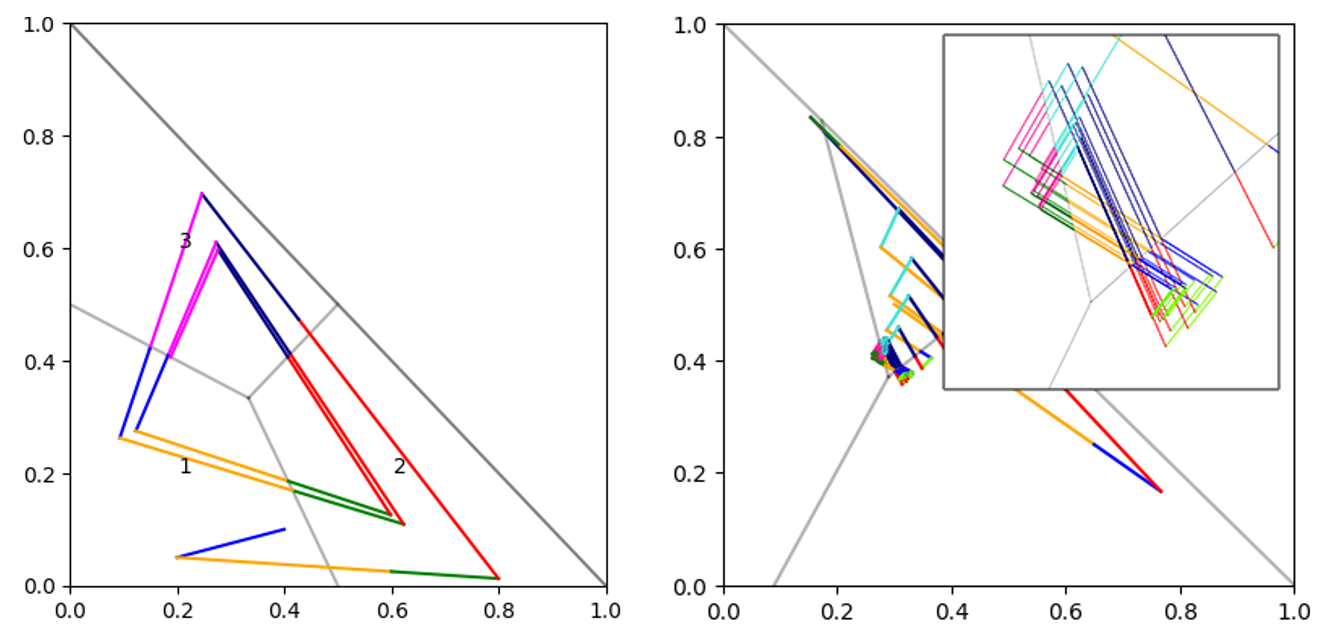}
	\caption{Dynamics of FP in Game (\ref{equ:Shpley-payoff}) and (\ref{equ:8-game}): We only show the evolution in $ \Delta_{A} $. The left figure is for Shapley game (\ref{equ:Shpley-payoff}) and the right figure is for game (\ref{equ:8-game}). Different colored lines represent different stage action profiles $ (a(t),b(t)) $ on the trajectory.}
	\label{fig:shapley-8}
\end{figure}

Although the dynamics in Figure \ref{fig:shapley-8} are quite different, note that both games (\ref{equ:Shpley-payoff}) and (\ref{equ:8-game}) have fully mixed NE. Hence a natural question arises: if FP does not converge, is it necessary that the NE has full support? And our main results partly answer this question. We will prove that when both players do not have fully mixed NE strategy, then FP converges to NE.

\section{Main Results and Methods}\label{sec3:result}
\subsection{Indifferent Point and Main Results}

Now we provide some basic notions used in our analysis. Since all the dynamics above (\ref{equ:DFP}, \ref{equ:cfp-def}, \ref{equ:cbr-def}) are driven by best response, it is important to figure out how it takes value.

\begin{definition}[\cite{ostrovski2014payoff}]\label{def:best-polygon}
	The best response region for Player A's $ i $-th action is the set 
	\begin{equation}\label{equ:best-res-A}
		Z^{A}_i=\{\mathbf{y}\in\Delta_{B}\mid (A \mathbf{y})_{i}\geq (A\mathbf{y})_{k}, \forall k=1,2,3\}.
	\end{equation}
	And the best response region for Player B's $ j $-th action is the set
	\begin{equation}\label{equ:best-res-B}
		Z^{B}_j=\{\mathbf{x}\in\Delta_{A}\mid (\mathbf{x}^{T}B)_{j}\geq (\mathbf{x}^{T} B)_{k}, \forall k=1,2,3\}.
	\end{equation}
\end{definition}

In a game, if some best response region $ Z^{k}_{i} $ is of measure zero, then the $ i $-th action is weakly dominated, and this game can be reduced to a lower-dimension game. Then according to Lemma \ref{lem:2xn}, it is trivial that FP converges. Hence in this paper, we suppose that all the best response regions have positive measure.

If the initial point $(\mathbf{x}(t_{0}),\mathbf{y}(t_{0}))$ belongs to $\operatorname{int}(Z^{B}_{i})\times\operatorname{int}(Z^{A}_{j})$, then $ \operatorname{BR}_{A}(\mathbf{y}(t)) $ and $ \operatorname{BR}_{B}(\mathbf{x}(t)) $ are constants until the trajectory leaves the $\operatorname{int}(Z^{B}_{i})$ or $\operatorname{int}(Z^{A}_{j})$. Then on this time interval, we can solve the CFP (\ref{equ:cfp-def}) with initial value $ \big(t_{0},(\mathbf{x}(t_{0}),\mathbf{y}(t_{0}))\big) $ as
$$
\left\{
\begin{array}{lr}
	\mathbf{x}(t) = (1-\frac{t_{0}}{t})\operatorname{BR}(\mathbf{y}(t))+\frac{t_{0}}{t}\mathbf{x}(t_{0}),\\
	\mathbf{y}(t) = (1-\frac{t_{0}}{t})\operatorname{BR}(\mathbf{x}(t))+\frac{t_{0}}{t}\mathbf{y}(t_{0}).
\end{array}
\right.
$$  
and BRD (\ref{equ:cbr-def}) as
$$
\left\{
\begin{array}{lr}
	\mathbf{x}(t) = (1-e^{t_{0}-t}) \operatorname{BR}(\mathbf{y}(t))+e^{t_{0}-t}\mathbf{x}(t_{0}),\\
	\mathbf{y}(t) = (1-e^{t_{0}-t}) \operatorname{BR}(\mathbf{x}(t))+e^{t_{0}-t}\mathbf{y}(t_{0}).
\end{array}
\right.
$$ 
We can see that both equations are linear combinations of two points, that is, $(\mathbf{x}(t),\mathbf{y}(t))=(1-s)a+s\cdot b$, $a=(\operatorname{BR}(\mathbf{y}),\operatorname{BR}(\mathbf{x}))$, $b=(x(t_{0}),y(t_{0}))$; for CFP, $s=\frac{t_{0}}{t}$, while for BRD, $s=e^{t_{0}-t}$. Hence, geometrically speaking, the only difference between CFP and BRD is their velocity, not the shape of trajectory. With an identical initial point, the trajectories of CFP and BRD would have the same direction and exit from the best-response region at the same place. 

On the boundary of the best response region, some of the inequalities in (\ref{equ:best-res-A}) are equal, i.e., there are more than one action being the best response to certain strategies. Such strategies form a segment, which is called an indifferent line and is defined formally as follows. 

\begin{definition}\label{def:indifferent_line}
	The indifferent lines of Player A and B are defined as
	\begin{equation}\label{equ:def_indif_line}
		\begin{aligned}
			l^{A}_{jk}&:=\{\mathbf{y}\in\Delta_{B}\mid (A\mathbf{y})_{j}=(A\mathbf{y})_{k}>(A\mathbf{y})_{i},\ \forall i\neq j,k\},\\
			l^{B}_{jk}&:=\{\mathbf{x}\in\Delta_{A}\mid (\mathbf{x}^{T}B)_{j}=(\mathbf{x}^{T}B)_{k}>(\mathbf{x}^{T}B)_{i},\ \forall i\neq j,k\}.
		\end{aligned}
	\end{equation}
\end{definition}

In the following we give an important lemma which states the existence and uniqueness of a special point on the generalized strategy space for any $ n\times n $ games.

\begin{lemma}\label{thm:existence-of-indiff-point}
	\begin{enumerate}
		\item For almost all $ n\times n $ normal form games $ (A,B) $, there exists an unique point $(\bar{\mathbf{x}},\bar{\mathbf{y}}) \in\{(\mathbf{x},\mathbf{y})\in\mathbb{R}^{2n}\mid\sum_{i=1}^{n}x_{i}=1,\sum_{i=1}^{n}y_{i}=1\}$ such that
		\[ \begin{aligned}
			\mathbf{x}^{T}A\bar{\mathbf{y}}=(\mathbf{x}^{\prime})^{T}A\bar{\mathbf{y}},\quad \forall \mathbf{x}, \mathbf{x}^{\prime}\in\Delta_{A},\\
			\bar{\mathbf{x}}^{T}A\mathbf{y}=\bar{\mathbf{x}}^{T}A\mathbf{y}^{\prime},\quad \forall \mathbf{y}, \mathbf{y}^{\prime}\in\Delta_{B}.
		\end{aligned}  \]
		\item Moreover, for $3\times 3$ games, if $ (\bar{\mathbf{x}},\bar{\mathbf{y}}) $ does not exist, the indifferent lines are parallel to each other.
	\end{enumerate}
\end{lemma}

The points $\bar{\mathbf{x}}, \bar{\mathbf{y}}$ in Lemma \ref{thm:existence-of-indiff-point} are called \emph{indifferent points}, at which the probabilities of some pure strategies can be negative, hence they can be within or outside the strategy simplex $\Delta_A, \Delta_B$. If $\bar{\mathbf{x}}\in\Delta_A\ (\bar{\mathbf{y}}\in \Delta_B) $, then we call it an Internal Indifferent Point (IIP). Based on Definition \ref{def:NE}, it is a NE if both $\bar{\mathbf{x}}$ and $\bar{\mathbf{y}}$ are IIPs. However, since FP may not converge to it as shown in Figure \ref{fig:shapley-8}, we distinguish the games into two classes.

\begin{definition}
	A game is called to be
	\begin{enumerate}
		\item  with IIP, if $ \bar{\mathbf{x}}\in \Delta_A,$ or $\bar{\mathbf{y}}\in\Delta_B $
		\item without IIP, if $ \bar{\mathbf{x}}\notin \Delta_A$ and $ \bar{\mathbf{y}}\notin\Delta_B $, including the case in which $\bar{\mathbf{x}}$ or $\bar{\mathbf{y}}$ does not exist.
	\end{enumerate}
\end{definition}

The proof of Theorem \ref{thm:existence-of-indiff-point} is based on the theory of linear systems of equations and can be found in \ref{app:existence-of-indiff-point}.
By the simulation in \ref{app:proportion}, by randomly sampling in the game space, it is easy to get a game without IIP, i.e., our class exists widely.

Now we can state our main results.

\begin{theorem}\label{thm:main-result}
Every CFP approaches equilibrium in every $3\times 3$ game without IIP.
\end{theorem} 

To prove Theorem \ref{thm:main-result}, we need to consider all possible scenarios of games without IIP. The scenarios are various due to the geometrical property of the game, including the location of indifferent points, the arrangements of each best response regions in $ \Delta_{A} $ and $ \Delta_{B} $. We separate the proof into two steps: the first is the game having unique NE in Section \ref{sec3-3:pure}, the second is the game having multiple NEs in Section \ref{sec3-5:other}. And We would present the basic technique used in next subsection.

\subsection{Projection Mapping}\label{sec3-2:project-function}
The basic idea in our analysis is to lower the dimension of FP dynamics from a four-dimensional product simplex $ \Delta_{A}\times\Delta_{B} $ to a plane.  First we have an observation below.

\begin{proposition}\label{pro:1}
	If $\bar{\mathbf{x}}\notin \Delta_{A}$ or $\bar{\mathbf{x}}$ does not exist, then there are at most two indifferent lines in $\Delta_{A}$.
\end{proposition}

Proposition \ref{pro:1} is straightforward since the best response region is convex and connected. And by Proposition \ref{pro:1}, we can obviously find that the two indifferent lines always intersect with one edge of $ \Delta_{A} $. For any point $ \mathbf{x}\in\Delta_{A} $, connect it to the indifferent point $ \bar{\mathbf{x}} $, then we get the intersection point $ \tilde{\mathbf{x}} $ of the line $ \overline{\mathbf{x}\bar{\mathbf{x}}} $ and the edge (Figure \ref{fig:project-map} shows more details). 

Now we can define the \textit{projection mapping} 
$$ \varphi_{A}:\Delta_{A}\to[0,1], \mathbf{x}\mapsto \tilde{\mathbf{x}}. $$ 
$ \varphi_{B} $ can be defined similarly. We underline that the product mapping is what we need, i.e.,
$$ \varphi = \varphi_{A}\times\varphi_{B}:\Delta_{A}\times\Delta_{B}\to[0,1]\times[0,1]=:\tilde{\Delta}. $$

\begin{figure}[htbp]
	\centering
	\includegraphics[width=0.8\linewidth]{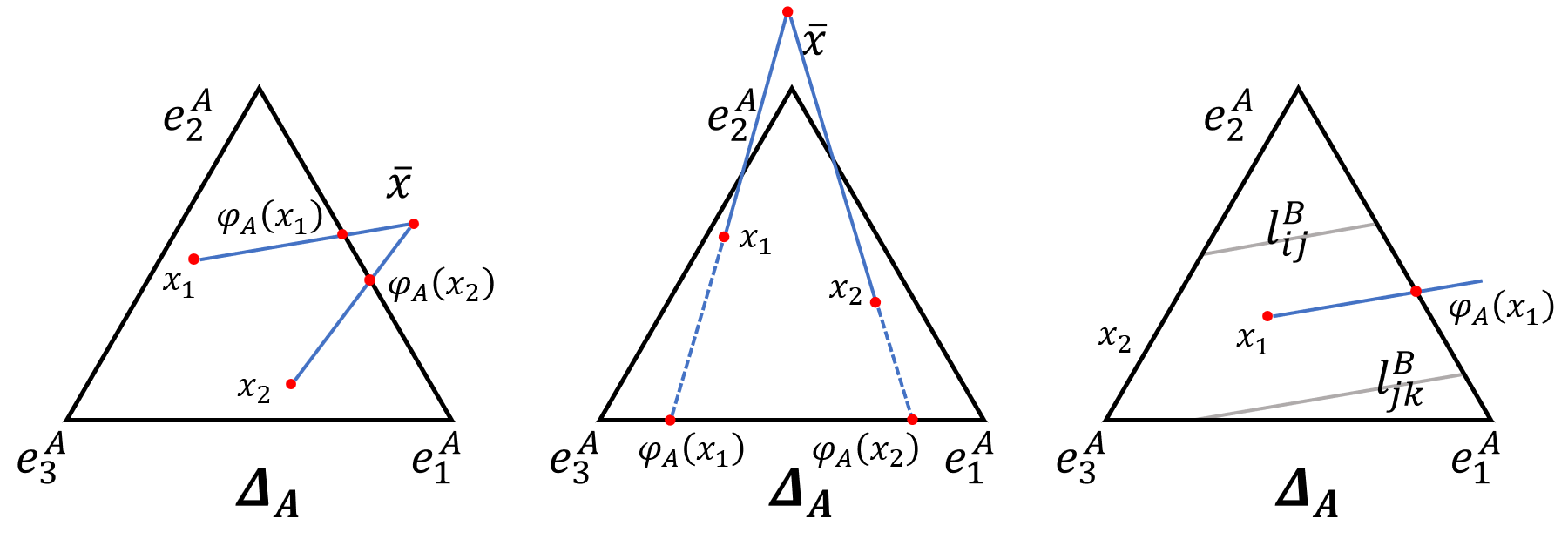}
	\caption{Three scenarios of $\varphi_{A}$: When $ \bar{\mathbf{x}} $ lies near an edge of $ \Delta_{A} $, directly connecting each point with $ \bar{\mathbf{x}} $ and the projection mapping $ \varphi_{A} $ can be well-defined. This case is shown by the first subgraph. When $ \bar{\mathbf{x}} $ lies near a vertex of $ \Delta_{A} $ as in the second subgraph, we need the extended line connecting each point with $ \bar{\mathbf{x}} $ in order to project all the points in $ \Delta_{A} $ onto the same edge. The third subgraph shows the case that $ \bar{\mathbf{x}} $ does not exist, then the projection should be parallel to the indifferent lines.}
	\label{fig:project-map}
\end{figure}

If $ \bar{\mathbf{x}} $ does not exist, by Theorem \ref{thm:existence-of-indiff-point}, the two indifferent lines are parallel to each other. Hence for any point $ x\in\Delta_{A} $, we could draw a line parallel to those indifferent lines. And we define $ \varphi_{A} $ to be the intersection point of this line and the edge of $ \Delta_{A} $. See the third subgraph of Figure \ref{fig:project-map}. 

Without loss of generality, below we assume that by $ \varphi_{A} $, all points in $ \Delta_{A} $ ($ \Delta_{B} $) are projected onto edge $ e_{1}^{A}e_{2}^{A} $ (edge $ e_{2}^{B}e_{3}^{B} $). And this does not affect our proof. 

Under $ \varphi $, the image of indifferent line $ l^{A}_{ij} $ would be one single point, while the image of best response region would be the segment between those points. For convenience, we denote $\varphi_{A}(l^{B}_{jk})=\tilde{l}^{B}_{jk}, \varphi_{A}({Z}_{j}^{B})= \tilde{Z}_{j}^{B},\ j,k=1,2,3 $. Furthermore, the $ \tilde{\Delta} $ would be divide into $ 9 $ cells, which is the product $ \tilde{Z}_{j}^{B}\times \tilde{Z}_{k}^{A} $, corresponding to different best response. For convenience, we label each cell with Roman numerals from \uppercase\expandafter{\romannumeral1}  to \uppercase\expandafter{\romannumeral9}. Figure \ref{fig:from-triangle-to-square} gives a clearer look of the projection mapping $ \varphi $. 

What we need to clarify next is the image of NE in $ \tilde{\Delta} $ under $ \varphi $. The image of a pure NE $ ({e}^{A}_{i},{e}^{B}_{j}) $ satisfies $(\tilde{e}^{A}_{i},\tilde{e}^{B}_{j})\in \tilde{Z}^{B}_{j}\times\tilde{Z}^{A}_{i}$. As for mixed NE, their location are summarized by Proposition \ref{pro:loc-NE}.

\begin{figure}[htbp]
	\centering
	\includegraphics[width=0.9\linewidth]{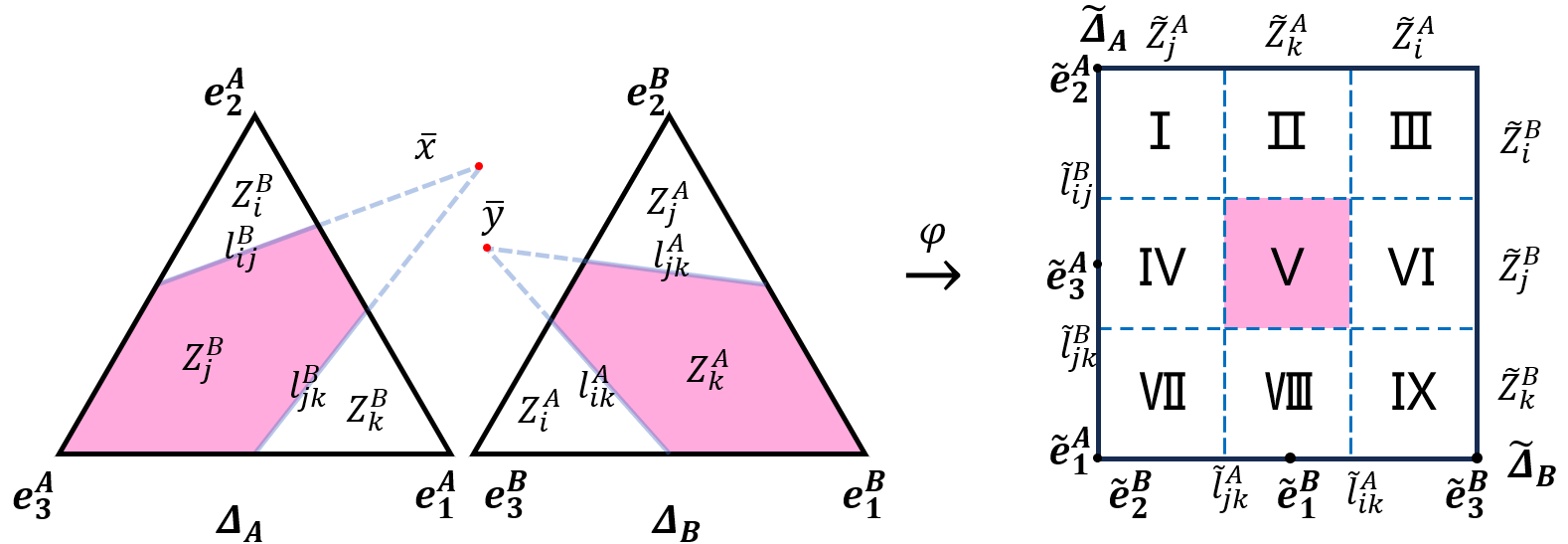}
	\caption{An example of projection mapping $ \varphi $: the whole strategy simplex $ \Delta_{A} $ is projected onto the edge $ e_{1}^{A}e_{2}^{A} $, and $ \Delta_{B} $ is projected onto the edge $ e_{2}^{B}e_{3}^{B} $, corresponding to $ \tilde{e}_{1}^{A}\tilde{e}_{2}^{A} $ and $ \tilde{e}_{2}^{B}\tilde{e}_{3}^{B} $ on the square $\tilde{\Delta}$. The vertex $ e^{A}_{3} $ is projected to a point $ \tilde{e}_{3}^{A} $ lying in the interior of $ \tilde{e}_{1}^{A}\tilde{e}_{2}^{A} $. In $ \Delta_{A} $ and $ \Delta_{B} $, the indifferent lines $ l^{A}_{jk}, l^{B}_{ij} $ are colored by blue, and they are projected into points $ \tilde{l}^{A}_{jk}, \tilde{l}^{B}_{ij} $ respectively on the edge of $\tilde{\Delta}$. Then the best response region $Z^{B}_{k}$ is projected onto a line segment $\tilde{Z^{B}_{k}}\subseteq  \tilde{e}_{1}^{A}\tilde{e}_{2}^{A}$. And the product of the best response region ${Z}^{B}_{j}\times{Z}^{A}_{k}$ is projected into a cell in $ \tilde{\Delta} $, and all of them are labeled from \uppercase\expandafter{\romannumeral1} to \uppercase\expandafter{\romannumeral9} respectively.}
	\label{fig:from-triangle-to-square}
\end{figure}

\begin{proposition}\label{pro:loc-NE}
	For the nondegenerate game $ (A,B) $ without IIP, under $ \varphi $, the image of its mixed NE must be one of four vertices of Cell \uppercase\expandafter{\romannumeral5}.
\end{proposition}
\begin{proof}
	Since $ \bar{\mathbf{x}}\notin\Delta_{A}, \bar{\mathbf{y}}\notin\Delta_{B} $, no mixed strategy can make opponent's three pure strategies to have the same payoff, i.e., there does not exist NE with support size $ 3\times3 $ or $ 2\times 3 $ according to Lemma \ref{lem:indifferent-of-support}. If a NE $ (e^{A}_{k},\mathbf{y}) $ has support size $ 1\times 2 $, then there are $ 2 $ pure strategies of Player B that is both best response against $ e^{A}_{k} $, which contradicts to non-degeneration, see Remark \ref{rmk:degenerate}. Hence the mixed NE can only have support size $ 2\times 2 $.
	
	For a NE $ (\mathbf{x},\mathbf{y}) $ with support size $ 2\times 2 $, suppose $ i,j\in\operatorname{supp}(\mathbf{x}) $, then according to Lemma \ref{lem:indifferent-of-support}, $ \mathbf{y}\in l^{A}_{ij} $. Similarly, suppose $ i^{\prime},j^{\prime}\in\operatorname{supp}(\mathbf{y}) $, then $ \mathbf{x}\in l^{A}_{i^{\prime}j^{\prime}} $. Hence under mapping $ \varphi $, $ \varphi(\mathbf{x},\mathbf{y})=(\tilde{l}^{B}_{ij},\tilde{l}^{A}_{i^{\prime}j^{\prime}}) $, which is a vertex of Cell \uppercase\expandafter{\romannumeral5}.
\end{proof}

Finally, we turn to the dynamics under mapping $ \varphi $. We can write down PBRD, the projected BRD, as 
\begin{equation}\label{prf1:equ1}
	\left\{
	\begin{array}{lr}
		\dot{\tilde{x}} = \tilde{e}_{k}^{A}-\tilde{x}(t),\\
		\dot{\tilde{y}} = \tilde{e}_{j}^{B}-\tilde{y}(t).
	\end{array}
	\right.\text{for }(\tilde{x},\tilde{y})\in \tilde{Z}^{B}_{j}\times\tilde{Z}^{A}_{k}, j,k=1,2,3.
\end{equation}

According to (\ref{prf1:equ1}), in Figure \ref{fig:from-triangle-to-square}, starting from each cell $\tilde{Z}^{B}_{j}\times\tilde{Z}^{A}_{k}$ in $\tilde{\Delta}$, the trajectory will move toward a point $\tilde{e}_{k}^{A}\times\tilde{e}_{j}^{B}$. Hence we can draw the vector field in each cell, indicating the direction of the PBRD trajectory, see Figure \ref{fig:provemixed}.

Proposition \ref{pro1:project=origin} below states that the convergence of BRD can be reduced to that of PBRD.
\begin{proposition}\label{pro1:project=origin}
	\begin{enumerate}
		\item $ \varphi(\mathbf{x},\mathbf{y})\in\tilde{Z}^{B}_{j}\times\tilde{Z}^{A}_{k} $, if and only if $ (\mathbf{x},\mathbf{y})\in {Z}^{B}_{j}\times{Z}^{B}_{k} $.
		\item If PBRD converges to the image of a certain NE under mapping $ \varphi $, then BRD converges to that NE.
	\end{enumerate}
\end{proposition}
\begin{proof}
	(1) is obvious. As for (2), according to the proof of Proposition \ref{pro:loc-NE}, the game would have only mixed NE with support size $ 2\times 2 $. If PBRD converges to it, according to Proposition \ref{pro:loc-NE}, the image of NE is one vertex of Cell \uppercase\expandafter{\romannumeral5}, then the trajectory of PBRD would eventually belongs to the $ 4 $ cells around NE. Hence in the long run, according to (1), for each player, the trajectory of BRD would only switch between $ 2 $ pure strategies, then Lemma \ref{lem:2xn} yields convergence of BRD. On the other hand, when PBRD converges to pure NE, then its trajectory would finally stay in one cell, and obviously BRD would converge.
\end{proof}

\subsection{Proof of Main Results, Part  \uppercase\expandafter{\romannumeral1}: Games with Unique NE} \label{sec3-3:pure}
Throughout this section we will prove that when a game without IIP incorporates only one NE, no matter it is pure or mixed, CFP would converge to it. During the proof, the key obstacle is that the trajectory may surround NE but never approaches to it, as Shapley game or game (\ref{equ:8-game}). But we will see that since $ \bar{\mathbf{x}}\notin\Delta_{A}, \bar{\mathbf{y}}\notin\Delta_{B} $, under the projection mapping, we are able to exclude such possibility and hence prove the convergence.

\begin{theorem}\label{thm:2-mixed}
	For any $ 3\times3 $ game without IIP which admits a unique NE,  CFP would converge to NE for any initial point, no matter the NE is pure or mixed.
\end{theorem}
\begin{proof}
Here we only prove the theorem when the NE is mixed. For the pure NE, although the method here is also feasible, \ref{app:PNE} gives another more direct proof. 

Consider all $ 3\times 3 $ games without IIP and having unique mixed NE, the location of indifferent lines and the arrangement of best response regions can have different on $ \Delta_{A}\times\Delta_{B} $, see the enumeration in \ref{app:enumeration}. Here we only prove the following typical case illustrated by Figure \ref{fig:provemixed}, and all the other cases are just similar. 

Now look at Figure \ref{fig:provemixed}, for analytical convenience, we give each point in $ [0,1]\times[0,1] $ a coordinate. $ \tilde{e}^{A}_{1}\times \tilde{e}^{B}_{2} $ is the original point $ (0,0) $,  $ \tilde{e}^{A}_{1}\times \tilde{e}^{B}_{3} $ is $ (1,0) $ and $ \tilde{e}^{A}_{3}\times \tilde{e}^{B}_{2} $ is $ (0,1) $. The width of each cell is $ p,q,r $ from left to right, $ p,q,r>0 $ and $ p+q+r=1 $, and the height of each cell is $ P,Q,R $ from up to down. The distance between $ \tilde{e}_{2}^{A} $ and $ \tilde{l}^{B}_{23} $ is $ M $, and the distance between $ \tilde{e}_{1}^{B} $ and $ \tilde{l}^{A}_{13} $ is $ m $. As a result, the only NE has a coordinate $ (p,Q+R) $, which lies in the upper left vertex of Cell \uppercase\expandafter{\romannumeral5}. 

By direct observation, starting from any initial point in any cells, at some time the trajectory would enter Cell \uppercase\expandafter{\romannumeral1} then leave it at its right edge $ \{(p,y)\mid Q+R< y\leq 1\} $, denoted by $ \mathfrak{l} $. Hence we could always assume that PBRD begins from $ \mathfrak{l} $. We denote the distance between initial point and NE by $ d_{0} $, $ 0< d_{0}\leq P $.

We note that starting from $\mathfrak{l}$ with $d_{0}$, the trajectory would follow three different kinds of path, depending on the value of $d_{0}$. To this end, when the trajectory moves into Cell \uppercase\expandafter{\romannumeral2}, since the best response is $ \tilde{e}^{A}_{1}\times \tilde{e}^{B}_{3}$, or $(1,0)$, the trajectory could enter Cell \uppercase\expandafter{\romannumeral3} or \uppercase\expandafter{\romannumeral5}. In order to enter Cell \uppercase\expandafter{\romannumeral3}, the distance $d_{1}$ of the trajectory from point $ (p+q,Q+R) $ on the left edge of Cell \uppercase\expandafter{\romannumeral3} should be greater than $0$. By solving similar triangles, we have 
\[ d_{1} = \phi_{0}(d_{0}) = \frac{r}{q+r}d_{0}-\frac{q}{q+r}(Q+R). \]
Hence to make sure $ d_{1} > 0 $, $ d_{0} $ should be large enough. Otherwise, the trajectory would enter \uppercase\expandafter{\romannumeral5}. As a result, the trajectory may follow a path \uppercase\expandafter{\romannumeral2} $\to$ \uppercase\expandafter{\romannumeral3} $\to$ \uppercase\expandafter{\romannumeral6} $\to$ \uppercase\expandafter{\romannumeral5} or directly from \uppercase\expandafter{\romannumeral2} to \uppercase\expandafter{\romannumeral5}. Similarly, when the trajectory is in Cell \uppercase\expandafter{\romannumeral5}, it can follow path \uppercase\expandafter{\romannumeral5} $\to$ \uppercase\expandafter{\romannumeral8} $\to$ \uppercase\expandafter{\romannumeral7} $\to$ \uppercase\expandafter{\romannumeral4} or \uppercase\expandafter{\romannumeral5} $\to$ \uppercase\expandafter{\romannumeral4}, still depending on different values of $d_{0}$.

In summary, the trajectory may follow three different kinds of paths as below:

\begin{enumerate}
	\item going through $8$ cells: \uppercase\expandafter{\romannumeral1}$ \to $ \uppercase\expandafter{\romannumeral2}$ \to $ \uppercase\expandafter{\romannumeral3}$ \to $ \uppercase\expandafter{\romannumeral6}$ \to $ \uppercase\expandafter{\romannumeral5}$ \to $ \uppercase\expandafter{\romannumeral8}$ \to $ \uppercase\expandafter{\romannumeral7}$ \to $ \uppercase\expandafter{\romannumeral4};
	\item going through $6$ cells: \uppercase\expandafter{\romannumeral1}$\to$ \uppercase\expandafter{\romannumeral2}$ \to $  \uppercase\expandafter{\romannumeral5}$ \to $ \uppercase\expandafter{\romannumeral8}$ \to $ \uppercase\expandafter{\romannumeral7}$ \to $ \uppercase\expandafter{\romannumeral4} or \uppercase\expandafter{\romannumeral1}$ \to $ \uppercase\expandafter{\romannumeral2}$ \to $ \uppercase\expandafter{\romannumeral3}$ \to $ \uppercase\expandafter{\romannumeral6}$ \to $ \uppercase\expandafter{\romannumeral5}$ \to $  \uppercase\expandafter{\romannumeral4};
	\item going through $4$ cells: \uppercase\expandafter{\romannumeral1} $ \to $ \uppercase\expandafter{\romannumeral2} $ \to $ \uppercase\expandafter{\romannumeral5} $ \to $ \uppercase\expandafter{\romannumeral4}.
\end{enumerate}
The latter two cases is convergent by Lemma \ref{lem:2xn}. We now study the first case.

For the first case, there are two real numbers $ d_{c}, d_{C}\in (0,P] $, such that: when $ d_{0}\in (0,d_{c}] $, $ d_{0} $ is small enough so that the trajectory passes 4 cells; when $ d_{0}\in(d_{C},A] $, the path would go though 8 cells; when $ d_{0}\in(d_{c},d_{C}] $, the trajectory would pass 6 cells. Without loss of generality, we assume its path is \uppercase\expandafter{\romannumeral2}$ \to $  \uppercase\expandafter{\romannumeral5}$ \to $ \uppercase\expandafter{\romannumeral8}$ \to $ \uppercase\expandafter{\romannumeral7}$ \to $ \uppercase\expandafter{\romannumeral4}$ \to $ \uppercase\expandafter{\romannumeral1}$ \to $ \uppercase\expandafter{\romannumeral2}.

We define a Poincar\'e mapping $ f:\mathfrak{l}\to\mathfrak{l} $. For any initial $ d_{0}\in(d_{C},P] $, after passing 8 cells, the trajectory would return to Cell \uppercase\expandafter{\romannumeral1} and leave it at $ f(d_{0})\in\mathfrak{l}$ again. And we use $ d_{i},i=0,\cdots,7 $ to capture the "distances" during the process, see Figure \ref{fig:provemixed}. Similarly, we can define another Poincar\'e mapping for the path going through 6 cells as $ \tilde{f}:\mathfrak{l}\to\mathfrak{l} $. According to the convergence of 6-cell path, we always have $ \tilde{f}(d_{0})<d_{0} $.

We note that $ f $ is composition of $ 8 $ different mappings, i.e., $ f(d_{0})=\phi_{7}\circ\phi_{6}\circ\phi_{5}\circ\phi_{4}\circ\phi_{3}\circ\phi_{2}\circ\phi_{1}\circ\phi_{0}(d_{0})$, which can be calculated based on similar triangles as below: 
\[ 	\begin{aligned}
	\frac{d_{1}+Q+R}{d_{0}+Q+R}=\frac{r}{q+r}\ &\Rightarrow\ d_{1} = \phi_{0}(d_{0})=\frac{r}{q+r}d_{0}-\frac{q}{q+r}(Q+R),\\
	\frac{d_{2}}{r}=\frac{d_{1}}{d_{1}+M}\ &\Rightarrow\  d_{2}=\phi_{1}(d_{1})=\frac{r\cdot d_{1}}{d_{1}+M},\\
	\frac{d_{3}}{d_{2}}=\frac{M}{d_{2}+p+q}\ &\Rightarrow\ d_{3}=\phi_{2}(d_{2})=\frac{M\cdot d_{2}}{d_{2}+p+q},\\
	\frac{Q-d_{3}}{q-d_{4}}=\frac{Q+R-d_{3}}{p+q}\ &\Rightarrow\ d_{4}=\phi_{3}(d_{3})=q-\frac{(p+q)(Q-d_{3})}{Q+R-d_{3}},\\
	\frac{d_{5}}{d_{4}}=\frac{R}{d_{4}+m}\ &\Rightarrow\ d_{5}=\phi_{4}(d_{4})=\frac{R\cdot d_{4}}{d_{4}+m},\\
	\frac{d_{6}}{d_{5}}=\frac{m}{d_{5}+P+Q}\ &\Rightarrow\ d_{6}=\phi_{5}(d_{5})=\frac{m\cdot d_{5}}{d_{5}+P+Q},\\
	\frac{p-d_{7}}{p-d_{6}}=\frac{P}{P+Q}\ &\Rightarrow\ d_{7}=\phi_{6}(d_{6})=\frac{P}{P+Q}d_{6}+\frac{Q}{P+Q}p,\\
	\frac{f(d_{0})}{d_{7}}=\frac{P}{d_{7}+q+r}\ &\Rightarrow\ f(d_{0})=\phi_{7}(d_{7})=\frac{P\cdot d_{7}}{d_{7}+q+r}.
\end{aligned} \]
Hence, 
\[ \begin{aligned}
	f&^{\prime}(d_{0})=\phi_{7}^{\prime}(d_{7})\cdot\phi^{\prime}_{6}(d_{6})\cdot\phi^{\prime}_{5}(d_{5})\cdot\phi^{\prime}_{4}(d_{4})\cdot\phi^{\prime}_{3}(d_{3})\cdot\phi^{\prime}_{2}(d_{2})\cdot\phi^{\prime}_{1}(d_{1})\cdot\phi^{\prime}_{0}(d_{0}), \\
	&=\frac{P(q+r)}{(d_{7}+q+r)^{2}}\cdot\frac{P}{P+Q}\cdot\frac{m(P+Q)}{(d_{5}+P+Q)^{2}}\cdot\frac{R\cdot m}{(d_{4}+m)^{2}}\cdot\frac{(p+q)R}{(Q+R-d_{3})^{2}}\\
	&\cdot\frac{M(p+q)}{(d_{2}+p+q)^{2}}\cdot\frac{r\cdot M}{(d_{1}+M)^{2}}\cdot\frac{r}{q+r}.
\end{aligned} \]
It is obvious that $ f^{\prime}(d_{0})>0 $. Note $ d_{k}\geq 0, k=0,\cdots,7 $ and $ d_{3}< Q $, we have
\[ f^{\prime}(d_{0})<\frac{P(q+r)}{(q+r)^{2}}\cdot (\frac{P}{P+Q})^{2}(\frac{r}{q+r})^{2}<1. \] 
Let $ g(d_{0})=d_{0}-f(d_{0}) $, we have $ \forall d_{0}\in (d_{C},P], g^{\prime}(d_{0})>0 $. Next we shall prove that for all $ \forall d_{0}\in (d_{C},P], g(d_{0})>0 $, then $ f(d_{0})<d_{0} $. Hence the Poincar\'e map is contractive.

\begin{figure}[h]
	\centering
	\includegraphics[width=0.8\linewidth]{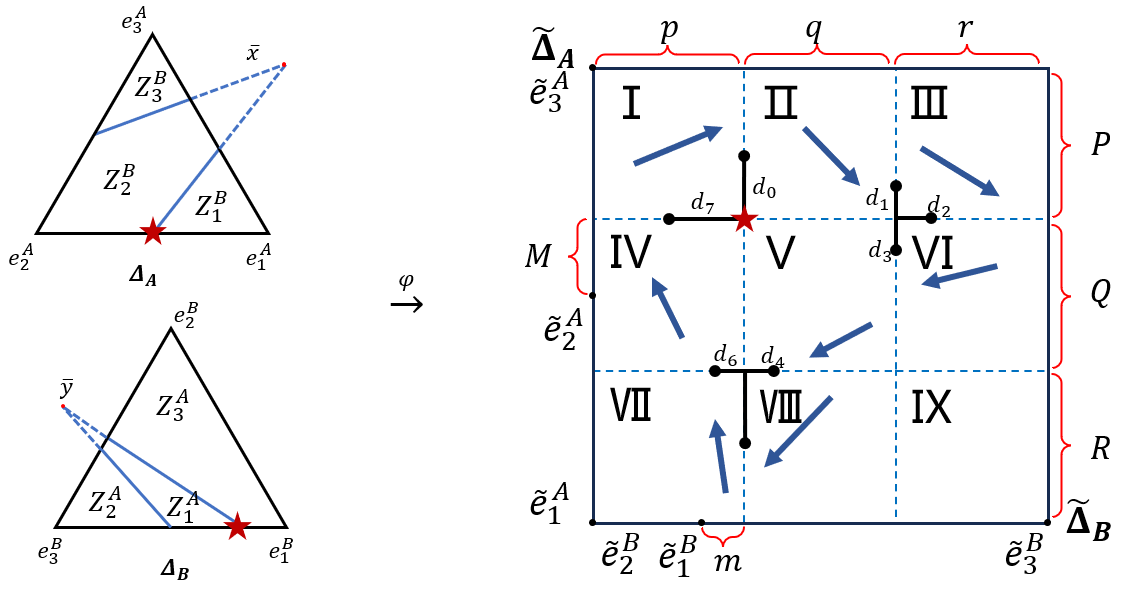}
	\caption{An example of the game without IIP with unique mixed NE: the red star profile in $\Delta_A$ and $\Delta_B$ is the unique mixed NE, and the blue arrow in each represents the vector field of PBRD.}
	\label{fig:provemixed}
\end{figure}

If $ d_{0}=d_{C} $, the trajectory would collapse from 8-cell path, \uppercase\expandafter{\romannumeral1}$ \to $ \uppercase\expandafter{\romannumeral2}$ \to $ \uppercase\expandafter{\romannumeral3}$ \to $ \uppercase\expandafter{\romannumeral6}$ \to $ \uppercase\expandafter{\romannumeral5}$ \to $ \uppercase\expandafter{\romannumeral8}$ \to $ \uppercase\expandafter{\romannumeral7}$ \to $ \uppercase\expandafter{\romannumeral4}, to 6-cell path, \uppercase\expandafter{\romannumeral1}$\to$ \uppercase\expandafter{\romannumeral2}$ \to $  \uppercase\expandafter{\romannumeral5}$ \to $ \uppercase\expandafter{\romannumeral8}$ \to $ \uppercase\expandafter{\romannumeral7}$ \to $ \uppercase\expandafter{\romannumeral4}, by crossing the point $ (p+q, Q+R) $ and avoiding entering the interior of Cell \uppercase\expandafter{\romannumeral3}. However, we can still view it as an 8-cell path with $ d_{1}=\phi_{0}(d_{C})=0, d_{2}=\phi_{1}(d_{C})= \frac{d\cdot 0}{0+E}=0, d_{3}=\phi_{2}(d_{C})=\frac{E\cdot0}{0+a+b}=0$. Hence $ f(d_{C}) $ is well-defined as the compose of $ \phi_{k},k=0,\cdots,7 $. 

As a result, $ f(d_{C})=\tilde{f}(d_{C}) <d_{C} $, we have $ g(d_{C})>0 $. Since $ g^{\prime}(d_{0})>0, \forall d_{0}\in (d_{C},A] $, we have $ \forall d_{0}\in (d_{C},A], g(d_{0})>0 $, i.e., $ f(d_{0})<d_{0} $. Hence the Poincar\'e mapping is contractive. 

We claim that for any initial point $ \exists\ d\in(d_{C},P] $, the trajectory would eventually satisfy $ f^{k}(d)<d_{C} $. Otherwise, suppose $\forall k, f^{k}(d)\geq d_{C} $. We get a sequence $\{ d^{k}=f^{k}(d)\} $ which is lower-bounded. Since the mapping $ f $ is contractive, $ \{d^{k}\} $ is strictly decreasing, obviously  $ \lim_{k\to\infty}d^{k}=d_{C} $. Hence we have, 
\[ d_{C}\leq \lim_{k\to\infty}f(d^{k})=f(d_{C})<d_{C}.\]
Contradiction. Thus after sufficient iterations, $ f^{k}(d) $ would be less than $ d_{C} $, and the trajectory will follow the 6-cell path and converge to NE. 
\end{proof}

We now give an example about the BRD and PBRD for games with only one pure NE. Consider the following games
\begin{example}\label{eg:proof_pure_NE}
	\begin{align}\label{game:unique_PNE}
		A = \begin{bmatrix}
			0      & 0.99      & -0.51      \\[1.5ex]
			0  & 0      & 0.5      \\[1.5ex]
			0.01      & 0.5      & 0
		\end{bmatrix},\quad B = \begin{bmatrix}
			0     & -1    & 1 \\[1.5ex]
			-0.99    & 0     & -2     \\[1.5ex]
			0.01     & 0    & 0
		\end{bmatrix}.
	\end{align}
\end{example}
Game (\ref{game:unique_PNE}) is without IIP and admits an unique NE $(e^A_3,e^B_1)$, which is a pure NE. Figure \ref{fig:pure} shows the trajectory of CFP in $\Delta_A\times\Delta_B$, and Figure \ref{fig:pure_PBRD} shows the trajectory of PBRD in $\tilde{\Delta}$. We can see that the best response regions $ Z^{B}_{1} $ and $ Z^{A}_{3} $ are quit narrow, and the trajectory approaches to them until it enters $ Z^{B}_{1}\times Z^{A}_{3} $. Once it enters, it will go to the pure NE straightly.

\begin{figure}[ht]
	\centering
	\begin{subfigure}[b]{0.69\textwidth}
		\includegraphics[width=0.87\linewidth]{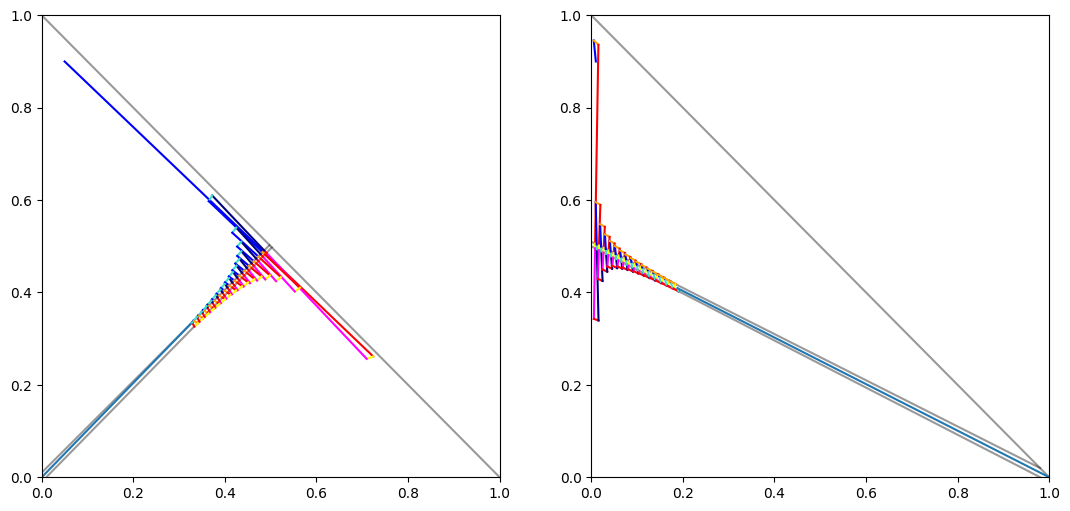}
		\caption{CFP in $\Delta_A\times \Delta_B$}
		\label{fig:pure}
	\end{subfigure}
	\hfill
	\begin{subfigure}[b]{0.3\textwidth}
		\includegraphics[width=1\linewidth]{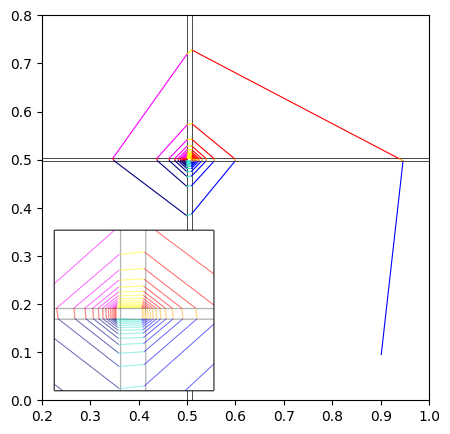}
		\caption{PBRD in $\tilde{\Delta}$}
		\label{fig:pure_PBRD}
	\end{subfigure}
	\caption{Trajectory of CFP and PBRD for Example \ref{eg:proof_pure_NE}: different colors indicate different action profiles. There are still nine cells in Figure \ref{fig:pure_PBRD} but Cell \uppercase\expandafter{\romannumeral2}, \uppercase\expandafter{\romannumeral4}, \uppercase\expandafter{\romannumeral5}, \uppercase\expandafter{\romannumeral6} and \uppercase\expandafter{\romannumeral8} are very narrow. The lower-left corner figure shows the enlargement of dynamics near the pure NE.}
	\label{fig:combine}
\end{figure}

\subsection{Proof of Main Results, Part  \uppercase\expandafter{\romannumeral2}: Games with Multiple NEs}\label{sec3-5:other}
For the game without IIP but having multiple NEs, the CFP dynamics near different NEs could be quite different. Next, we introduce some concepts to distinguish NE, which is necessary for the analysis.

For a mixed NE $ (\mathbf{x},\mathbf{y}) $ with support size $ 2\times 2 $, without loss of generality, suppose $ \mathbf{x}\in e^{A}_{i}e^{A}_{j} $ and $ \mathbf{y}\in e^{B}_{i^{\prime}}e^{B}_{j^{\prime}} $, then based on the lemma \ref{lem:indifferent-of-support}, $\mathbf{x}\in l^{B}_{i^{\prime}j^{\prime}}, \mathbf{y}\in l^{A}_{ij}$.
\begin{definition}\label{def:saddle-sink}We say NE $ (\mathbf{x},\mathbf{y}) $ above is a 
	\begin{enumerate}
		\item {saddle}, if there exist neighborhoods $U_{A}$ of $\mathbf{x}$ and $U_{B}$ of $ \mathbf{y} $ such that 
		\begin{equation}\label{equ:def_saddle}
			\overline{\mathbf{x}e^{A}_{i}}\cap U_{A} \subseteq Z^{B}_{i^{\prime}}, \overline{\mathbf{x}e^{A}_{j}}\cap U_{A}\subseteq Z^{B}_{j^{\prime}},
			\overline{\mathbf{y}e^{B}_{i^{\prime}}}\cap U_{B} \subseteq Z^{A}_{i}, \overline{\mathbf{y}e^{B}_{j^{\prime}}}\cap U_{A}\subseteq Z^{A}_{j}; 
		\end{equation}
		\item {sink}, if there exist neighborhoods $U_{A}$ of $\mathbf{x}$ and $U_{B}$ of $ \mathbf{y} $ such that 
		\begin{equation}\label{equ:def_sink}
			\overline{\mathbf{x}e^{A}_{i}}\cap U_{A}\subseteq Z^{B}_{i^{\prime}}, \overline{\mathbf{x}e^{A}_{j}}\cap U_{A}\subseteq Z^{B}_{j^{\prime}}, \overline{\mathbf{y}e^{B}_{i^{\prime}}}\cap U_{B} \subseteq Z^{A}_{j}, \overline{\mathbf{y}e^{B}_{j^{\prime}}}\cap U_{A}\subseteq Z^{A}_{i}. 
		\end{equation}
	\end{enumerate}
\end{definition}

First, Definition \ref{def:saddle-sink} has already considered all the possible arrangements of best response regions, thus a mixed NE is either a saddle or a sink. 

Second, we can draw the vector field near each mixed NE in $\tilde{\Delta}$, see Figure \ref{fig:two-class-2x2-ne}. For the saddle NE, by Definition \ref{def:saddle-sink}, $\overline{\mathbf{x}e^{A}_{i}}\cap U_{A} \subseteq Z^{B}_{i^{\prime}}$,$\overline{\mathbf{y}e^{B}_{i^{\prime}}}\cap U_{B} \subseteq Z^{A}_{i}$, thus on upper-left cell near the saddle NE, the best response of Player B is $\tilde{e}^B_{i^{\prime}}$ and the best response of Player A is $\tilde{e}^A_{i}$. Hence the vector field would be toward the point $\tilde{e}^A_{i}\times\tilde{e}^B_{i^{\prime}}$, i.e. points to the upper left. Similarly, we can draw all the other vector fields in the other three cells. And we can see that starting from two cells along one diagonal, the trajectory would approach the saddle NE, while it would move away from the saddle NE starting from the other two cells along the other diagonal. Thus the dynamics near the saddle NE is quite similar to the dynamics near the saddle point in the smooth dynamical systems, which is defined base on the Jacobian matrix. Hence it is reasonable to call it a saddle. It is the same for the definition of sink NE.

\begin{figure}[htbp]
	\centering
	\includegraphics[width=0.8\linewidth]{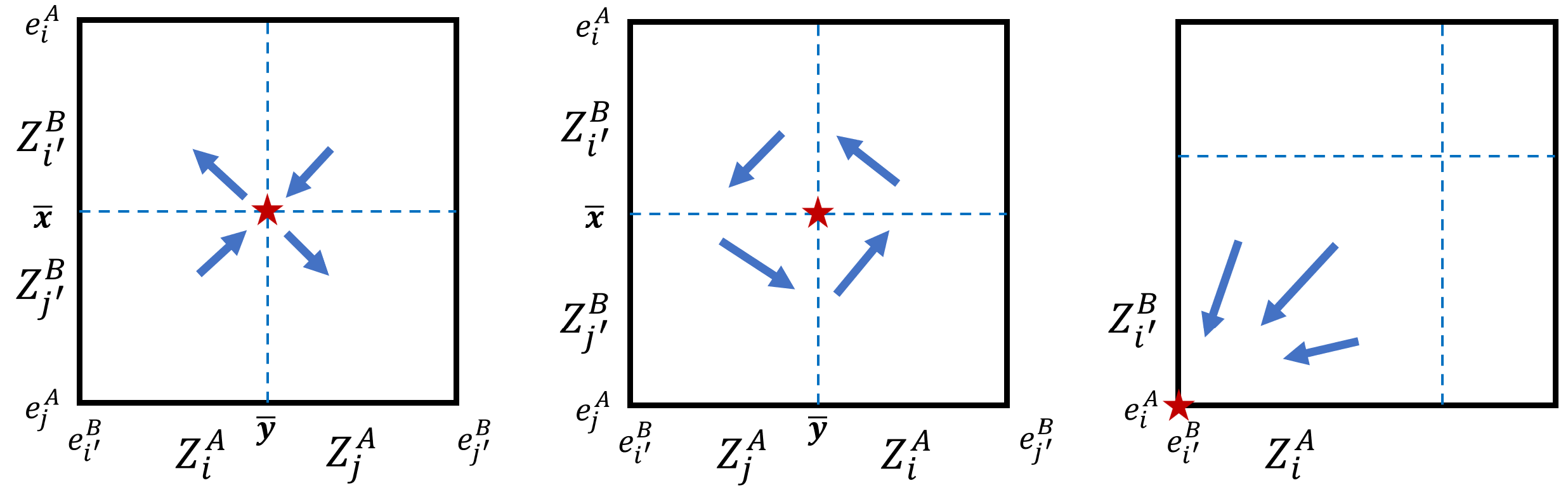}
	\caption{The dynamics near saddle NE and sink NE under mapping $\varphi$.}
	\label{fig:two-class-2x2-ne}
\end{figure}

By Definition \ref{def:saddle-sink}, the previous convergence results Theorem \ref{thm:2-mixed} are dealing with the sink NE. However, when it comes to multiple NEs, it is possible that there is a saddle NE. Proposition \ref{pro:exist-saddle} below states that this claim is confirmative.

\begin{proposition}\label{pro:exist-saddle}
	Suppose the $ 3\times 3 $ game without IIP has multiple NEs, then there must exist a mixed NE being a saddle, and a NE being a sink. 
\end{proposition}
\begin{proof}
Without loss of generality, we suppose the game has the same shape as Figure \ref{fig:from-triangle-to-square}, in which the vertices of $\tilde{\Delta}$ and the location of indifferent lines are fixed, but to what pure action each best response region corresponds is still  uncertain. 

Use contradiction. Suppose that the game has multiple NEs but no saddle NE. We shall start from one fixed NE, and prove with the absence of saddle, the game would have unique NE, and hence get contradiction. 

(1) First, assume that the game has a pure NE lying in Cell \uppercase\expandafter{\romannumeral5}.

Then by Figure \ref{fig:from-triangle-to-square} it is $e^{A}_{3}\times e^{B}_{1}$. Hence $ e^{A}_{3} $ is the best response to $ e^{A}_{3} $, and vice verse. Thus Cell \uppercase\expandafter{\romannumeral5} is the product $\tilde{Z}^{B}_{1}\times\tilde{Z}^{A}_{3}$. Since Cell \uppercase\expandafter{\romannumeral5} is fixed, we know that $e^{B}_{i}\notin Z^{A}_{3}$ for $i\neq1$ and $e^{A}_{j}\notin Z^{B}_{1}$ for $j\neq 3$. Thus any pure strategy profile $e^{A}_{3}\times e^{B}_{i},i\neq1$ or $e^{A}_{j}\times e^{B}_{1},j\neq3$ is not a NE. 

As for the other pure strategy profiles, we claim that they are not NE either. Otherwise, suppose $e^{A}_{1}\times e^{B}_{2}$ is a NE. Then we have $e^{A}_{1}\in Z^{B}_{2}$ and $e^{B}_{2}\in Z^{A}_{1}$. Then we can fix the arrangement of BRPs in $ \Delta_{A} $ and $ \Delta_{B} $. 
Then from the figure, we know that the profile $(l^{B}_{12}\cap e^{A}_{1}e^{A}_{3},l^{A}_{13}\cap e^{B}_{1}e^{B}_{2})$  is a mixed NE and a saddle by Definition \ref{def:saddle-sink}, which contradicts the absence of saddle. Using the same logic, we can prove all the pure strategy profiles cannot be NE.

On the other hand, we claim that any mixed strategy profile cannot be a NE either. Otherwise, we suppose the mixed strategy $ \mathbf{x} $ on edge $e^{A}_{2}e^{A}_{3}$ and mixed strategy $ \mathbf{y} $ on edge $e^{B}_{2}e^{B}_{1}$ constitute a NE. Then they must be the intersection points corresponding to $l^{B}_{12}$ and $l^{A}_{23}$. Since we already have $e^{A}_{3}\in{Z}^{B}_{1}$ and $e^{B}_{1}\in{Z}^{A}_{3}$, by Definition \ref{def:saddle-sink}, $ (\mathbf{x},\mathbf{y}) $ is a saddle. Contradiction. 

Thus the game has an unique NE, contradiction.

(2) Next we can assume that the game has a pure NE which lies on the vertex of $ \tilde{\Delta} $.

(3) Finally, we assume that the game has a mixed NE which is a sink.

The analysis for case (2) and (3) is similar to case (1), see \ref{app:exist-saddle} for details. In all the cases, the absence of saddle NE and the existence of multiple NEs would lead to a contradiction. Thus we prove the existence of saddle NE.

The existence of sink NE is a straightforward corollary of our main results Theorem \ref{thm:main-result}, whose proof only relies on the existence of saddle NE. Hence this does not lead to circular argument. And we prove the whole proposition.
\end{proof}

\begin{theorem}\label{pro:converge_multi_NE}
	For any $ 3\times3 $ game without IIP which admits multiple NEs,  CFP would converge to some NE for any initial point.
\end{theorem}
\begin{proof}
By Proposition \ref{pro:exist-saddle}, the game has a saddle NE $(\mathbf{x},\mathbf{y})\in\Delta$, which must be mixed. We let $\mathbf{x}\in e^{A}_{i}e^{A}_{j}$ and $\mathbf{y}\in e^{B}_{i^{\prime}}e^{B}_{j^{\prime}}$. Then under the projection mapping $\varphi$, by Lemma \ref{lem:indifferent-of-support}, we have $\varphi(\mathbf{x},\mathbf{y})=\tilde{l}^{B}_{i^{\prime}j^{\prime}}\times \tilde{l}^{A}_{ij}$, where $\tilde{l}^{B}_{i^{\prime}j^{\prime}}$ lies between $\tilde{e}^{A}_{i}$ and $\tilde{e}^{A}_{j}$ in $\tilde{\Delta}_{A}$, while $\tilde{l}^{A}_{ij}$ lies between $\tilde{e}^{B}_{i^{\prime}}$ and $\tilde{e}^{B}_{j^{\prime}}$ in $\tilde{\Delta}_{B}$. 

By Proposition \ref{pro:loc-NE}, the point $ \tilde{l}^{B}_{i^{\prime}j^{\prime}}\times \tilde{l}^{A}_{ij} $ is one vertex of Cell \uppercase\expandafter{\romannumeral5}. Without loss of generality, we assume it is the upper left vertex of Cell \uppercase\expandafter{\romannumeral5}; meanwhile, we assume that in $\tilde{\Delta}_{A}$, $\tilde{e}^{A}_{i}$ is above $\tilde{l}^{B}_{i^{\prime}j^{\prime}}$, $\tilde{e}^{A}_{j}$ is below $\tilde{l}^{B}_{i^{\prime}j^{\prime}}$, while in $\tilde{\Delta}_{B}$, $\tilde{e}^{B}_{i^{\prime}}$ is on the left of $\tilde{l}^{A}_{ij}$, and $\tilde{e}^{B}_{j^{\prime}}$ is on the right of $\tilde{l}^{A}_{ij}$.

For the four cells around the saddle NE $ \tilde{l}^{B}_{i^{\prime}j^{\prime}}\times \tilde{l}^{A}_{ij} $, i.e., Cell \uppercase\expandafter{\romannumeral1}, \uppercase\expandafter{\romannumeral2}, \uppercase\expandafter{\romannumeral4}, \uppercase\expandafter{\romannumeral5}, by Definition \ref{def:saddle-sink}, there are only two possible arrangements of BRPs: (1) Cell \uppercase\expandafter{\romannumeral1} is $\tilde{Z}^{B}_{i^{\prime}}\times \tilde{Z}^{A}_{i}$; (2) Cell \uppercase\expandafter{\romannumeral1} is $\tilde{Z}^{B}_{j^{\prime}}\times \tilde{Z}^{A}_{j}$. And once Cell \uppercase\expandafter{\romannumeral1} is fixed, the other three cells are also determined, see Figure \ref{fig:saddlepoint}. Below we prove the theorem for case (1), and the proof for case (2) is similar.

For case (1), i.e., Cell \uppercase\expandafter{\romannumeral1} is $\tilde{Z}^{B}_{i^{\prime}}\times \tilde{Z}^{A}_{i}$, the trajectory starting from this cell would go toward $\tilde{e}^{A}_{i}\times \tilde{e}^{B}_{i^{\prime}}$, which lies in the upper left of the saddle NE, hence the trajectory would diverge from the NE. Connect the saddle NE with $\tilde{e}^{A}_{i}\times \tilde{e}^{B}_{i^{\prime}}$ in Cell \uppercase\expandafter{\romannumeral1}, we can get a line segment which divides Cell \uppercase\expandafter{\romannumeral1} into two parts: the trajectory starting from one side of the line cannot cross it, see the yellow line in Cell \uppercase\expandafter{\romannumeral1} from Figure \ref{fig:saddlepoint}. Similarly, we can get the other yellow line in Cell \uppercase\expandafter{\romannumeral5}.

On the other hand, the trajectory starting from Cell \uppercase\expandafter{\romannumeral2} ( which is $\tilde{Z}^{B}_{i^{\prime}}\times \tilde{Z}^{A}_{j}$)  would move toward the point  $\tilde{e}^{A}_{j}\times \tilde{e}^{B}_{i^{\prime}}$, which lies in the lower left of the saddle NE. Hence the trajectory would move closer to NE. We can still connect NE with $\tilde{e}^{A}_{j}\times \tilde{e}^{B}_{i^{\prime}}$ and extend the line (the dashed blue line in Cell \uppercase\expandafter{\romannumeral4} from Figure \ref{fig:saddlepoint}) to Cell \uppercase\expandafter{\romannumeral2}. And we get a line segment in Cell \uppercase\expandafter{\romannumeral2}, from which the trajectory would converge to the saddle NE, see the solid blue line in Cell \uppercase\expandafter{\romannumeral2} from Figure \ref{fig:saddlepoint}. We note that this line segment also divides Cell \uppercase\expandafter{\romannumeral2} into two parts, and the trajectory in one side cannot cross the blue line, either. Similarly, we can get the other solid blue line in Cell \uppercase\expandafter{\romannumeral4}.

\begin{figure}[htbp]
	\centering
	\includegraphics[width=0.8\linewidth]{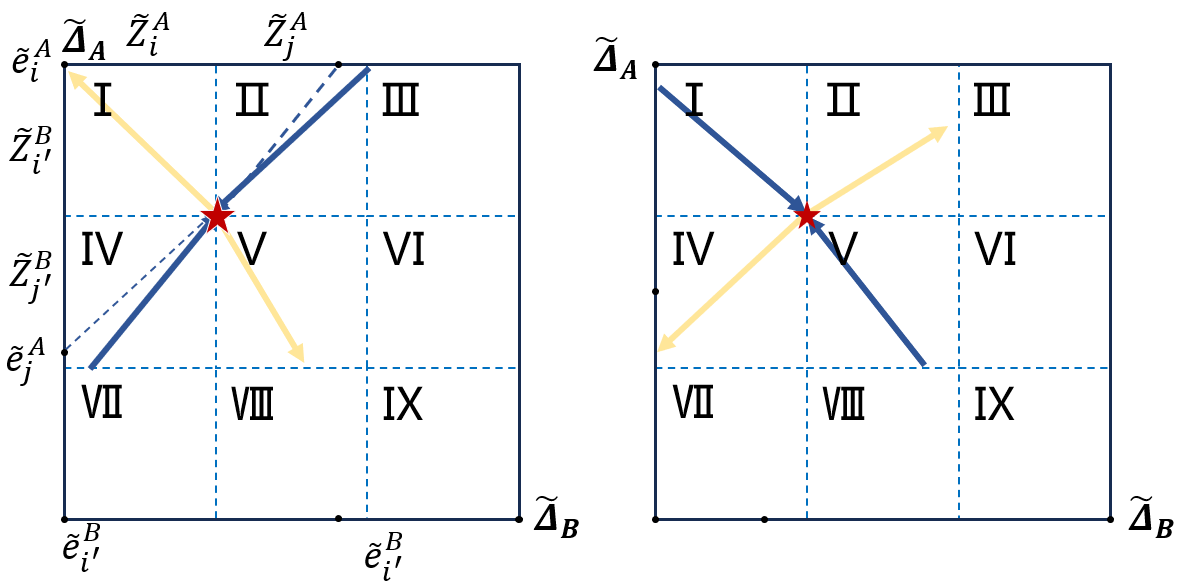}
	\caption{Two cases of saddle NE: in the left subgraph, in Cell \uppercase\expandafter{\romannumeral1} and \uppercase\expandafter{\romannumeral5} the trajectory would move away from the saddle NE; in Cell \uppercase\expandafter{\romannumeral2} and \uppercase\expandafter{\romannumeral4}, the trajectory moves toward the saddle NE. The right subgraph is contrary to the left.}
	\label{fig:saddlepoint}
\end{figure}

Now, suppose that there exists an initial point from which the trajectory is not convergent. Then by Lemma \ref{lem:2xn}, along the trajectory each player will play all the three action infinitely. 

Consider all the locations of initial point below.

(1) If the initial point belongs to the upper left area of the solid blue lines, which is also inside Cell \uppercase\expandafter{\romannumeral1}, \uppercase\expandafter{\romannumeral2}, \uppercase\expandafter{\romannumeral4}, by the direction of the trajectory, it would stay in this area, hence along the trajectory, each player would play at most two action infinitely.

(2) We study the case when initial point is in the other area. 

If the trajectory enters Cell \uppercase\expandafter{\romannumeral4}, since it cannot cross the solid blue line, the trajectory must enter Cell \uppercase\expandafter{\romannumeral5}.

If it enters Cell \uppercase\expandafter{\romannumeral5}, since it cannot cross the yellow line, it must enter Cell \uppercase\expandafter{\romannumeral8}.

Starting from Cell \uppercase\expandafter{\romannumeral8}, in order to excite all the three actions for each player, the trajectory needs to enter Cell \uppercase\expandafter{\romannumeral9}. 

For the same reason, 
the trajectory from Cell \uppercase\expandafter{\romannumeral9} needs to enter Cell \uppercase\expandafter{\romannumeral6};  from Cell \uppercase\expandafter{\romannumeral6} it needs to enter Cell \uppercase\expandafter{\romannumeral3}; from Cell \uppercase\expandafter{\romannumeral3} it needs to enter Cell \uppercase\expandafter{\romannumeral2}.

From Cell \uppercase\expandafter{\romannumeral2}, the trajectory must enter Cell \uppercase\expandafter{\romannumeral5}, by the direction. Now the trajectory is located in the upper right side of the yellow line in Cell \uppercase\expandafter{\romannumeral5}. Since it could not cross the yellow line, the trajectory can only enter Cell \uppercase\expandafter{\romannumeral6}, \uppercase\expandafter{\romannumeral8} or \uppercase\expandafter{\romannumeral9}. By the analysis above, the trajectory from any cell of them would again return to Cell \uppercase\expandafter{\romannumeral5}.

Hence in the long run, the trajectory would stay in at most 6 cells on the right. Thus one player can only play at most two actions.  

Therefore, we have proved that from all the initial points, at least one player would not play all his three actions along the trajectory infinitely. And we get our conclusion. 
\end{proof}

Combining Theorem \ref{thm:2-mixed} and Proposition \ref{pro:converge_multi_NE}, we can finally get our main result as Theorem \ref{thm:main-result}.

Moreover, Proposition \ref{pro:converge_multi_NE} tells us more about the property of NE. In fact, we could accordingly judge whether the convergent point is saddle or sink just based on the trajectory of FP.
\begin{corollary}
	For the game without IIP, after long enough steps, if 
	\begin{enumerate}
		\item FP trajectory only has one action profile, then $ (\mathbf{x}(t),\mathbf{y}(t)) $ converges to a pure NE.
		\item FP trajectory switches between two action profiles, then $ (\mathbf{x}(t),\mathbf{y}(t)) $ converges to a saddle NE.
		\item FP trajectory switches among four action profiles, then $ (\mathbf{x}(t),\mathbf{y}(t)) $ converges to a sink NE.
	\end{enumerate}
\end{corollary}
When there are multiple NEs, the initial point would determine to which one FP would converge. We note that the second case happens with lowest probability. In fact, it requires that the initial point should belong to one of the stable manifolds and those two manifolds should lie on the same straight line. Hence  not only the initial point takes zero measurement in $ \tilde{\Delta} $, but the payoff matrix itself also takes zero measurement in the space of $ 3\times 3 $ matrices. 

\section{Further Discussion}\label{sec4:further}
\subsection{Comparison}\label{sec4.1}
To the best of our knowledge, the latest class of $ 3\times 3 $ games with FPP is supermodular/quasi-supermodular (\cite{bergerTwoMoreClasses2007}). In this section, we will compare our convergent class with it. It turns out that there exists a game with $\bar{\mathbf{x}}\notin\Delta_A, \bar{\mathbf{y}}\notin\Delta_B$ which is not quasi-supermodular, and there exists a quasi-supermodular game, whose indifferent point belongs to mixed-strategy space. 

\begin{definition}[\cite{bergerTwoMoreClasses2007}]\label{def:supermodular}
	A game $(A,B)$ is called
	\begin{enumerate}
		\item quasi-supermodular, if for all
		$i< i^{\prime }$ and $j< j^{\prime }$:
		$$a_{i'j}>a_{ij}\Longrightarrow a_{i'j'}>a_{ij'}\quad and\quad b_{ij'}>b_{ij}\Longrightarrow b_{i'j'}>b_{i'j}.$$
		\item supermodular, if for all $i< i^{\prime }$ and $j< j^{\prime }:$
		$$a_{i'j'}-a_{ij'}>a_{i'j}-a_{ij}\quad and\quad b_{i'j'}-b_{i'j}>b_{ij'}-b_{ij}.$$
	\end{enumerate}
\end{definition}

\begin{example}\label{eg:quasi-}
	\[ A=\begin{bmatrix}
		\frac{1}{3} & \frac{1}{3} & 0\\[1.5ex]
		0 & \frac{2}{3} & \frac{2}{3}\\[1.5ex]
		-\frac{2}{3} & 0 & 1
	\end{bmatrix},  B=\begin{bmatrix}
		0 & \frac{2}{3} & \frac{1}{3}\\[1.5ex]
		-\frac{2}{3} & 0 & \frac{1}{3}\\[1.5ex]
		1 & \frac{2}{3} & 0
	\end{bmatrix}. \]
\end{example}
We can see the indifferent point of $ (A,B) $ is $\bar{\mathbf{x}}=(-\frac{1}{2}, \frac{5}{6}, \frac{2}{3}), \bar{\mathbf{y}}=(\frac{5}{6}, -\frac{1}{2}, \frac{2}{3})^{T} $. Hence it falls into the class of games without IIP identified in this paper. But according to Definition \ref{def:supermodular}, $A$ satisfies the condition of quasi-supermodular game but $B$ does not. 

However, as we mentioned before, even the payoff matrices of two games having totally different elements, they could induce the same dynamics under FP with identical initial point. To give a more strict criterion, we consider two equivalence concepts. The first is linearly equivalent, which is often used in analyzing the game dynamics (\cite{ostrovski2014payoff,chen2022convergence}).

\begin{definition}(\cite{ostrovski2014payoff})
	Two $m\times n$ games $( A, B)$ and $(\tilde{A},\tilde{B})$ are linearly equivalent, if for $ i=1,2,\cdots,m,\ j=1,2,\cdots,n $ the payoff matrix $ A=(a_{ij}) $ and $\tilde{A}=(\tilde{a}_{ij})$ for Player A satisfy
	\[ \tilde{a}_{ij}=c\cdot a_{ij}+c_{j}, \]
	and the payoff matrix $ B=(b_{ij}) $ and $\tilde{B}=(\tilde{b}_{ij})$ for Player B satisfy
	\[ \tilde{b} _{ij}= d\cdot b_{ij}+ d_i, \]
	where $ c,d $ are positive, and $ c_{j},d_{i}\in\mathbb{R} $.
\end{definition} 

The second equivalence relationship is to exchange the rows or columns of payoff matrices $(A,B)$ simultaneously. It is also natural, and we have the following notation.
\begin{definition}
	For $m\times n$ games $(A,B)$, given permutations $\sigma:\{1,\cdots,n\}\to \{1,\cdots,n\}$ and $\tau:\{1,\cdots,m\}\to\{1,\cdots,m\}$, $(A^{\prime},B^{\prime})$ is said to be equivalent to $(A,B)$ under $(\sigma,\tau)$, if
	\[ A_{ij} = A^{\prime}_{\sigma(i)\tau(j)}, B_{ij} = B^{\prime}_{\sigma(i)\tau(j)}. \]
\end{definition}

It is very easy to check that 
\begin{proposition}
	If two games $(A,B)$ and $(A^{\prime}, B^{\prime})$ are linearly equivalent, then
	\begin{enumerate}
		\item $ \operatorname{BR}_{A}=\operatorname{BR}_{A^{\prime}},\ \operatorname{BR}_{B}=\operatorname{BR}_{B^{\prime}}$.
		\item $(A,B)$ is (quasi-)supermodular iff $(A^{\prime}, B^{\prime})$ is (quasi-)supermodular.
	\end{enumerate}
\end{proposition}

As a result, if we have found a game without IIP which is not quasi-supermodular, then any game which is linearly equivalent to it is not quasi-supermodular, either. So we only need to consider the second equivalence. By enumeration, there are only two permutations that can make $B$ satisfy the condition in Definition \ref{def:supermodular}. The first is to take $\sigma(1,2,3)=(3,1,2)$ to exchange the row, and under $\sigma$, we have
\[ A_{\sigma} = \begin{bmatrix}
-\frac{2}{3} & 0 & 1\\[1.5ex]
\frac{1}{3} & \frac{1}{3} & 0\\[1.5ex]
0 & \frac{2}{3} & \frac{2}{3}\\
\end{bmatrix}, \quad B_{\sigma} = \begin{bmatrix}
1 & \frac{2}{3} & 0\\[1.5ex]
-\frac{2}{3} & 0 & \frac{1}{3}\\[1.5ex]
0 & \frac{2}{3} & \frac{1}{3}\\
\end{bmatrix}. \]
The other is to take $\sigma(1,2,3)=(2,1,3)$ to exchange the row, and to take $\tau(1,2,3)=(3,2,1)$ to exchange the column, then we have 
\[ A_{\sigma\tau}=\begin{bmatrix}
	\frac{2}{3} & \frac{2}{3} & 0\\[1.5ex]
	0 & \frac{1}{3} & \frac{1}{3}\\[1.5ex]
	1 & 0 & -\frac{2}{3}\\
\end{bmatrix},  
B_{\sigma\tau}=\begin{bmatrix}
	\frac{1}{3} & 0 &-\frac{2}{3} \\[1.5ex]
	\frac{1}{3} & \frac{2}{3} & 0\\[1.5ex]
	0 & \frac{2}{3} & 1
\end{bmatrix}. \]
We can check that both $A_{\sigma}$ and $A_{\sigma\tau}$ do not satisfy the condition of quasi-supermodular. Hence Example \ref{eg:quasi-} demonstrates that the class of games without IIP is indeed different from the game class with FPP identified in the latest work.

\subsection{Degenerate Game}
So far we only consider the nondegenerate game, which means for any pure strategy of each player there is a unique best response. Although almost all games are nondegenerate, and the analysis about FP mainly constrain on this kind in the literature (\cite{bergerFictitiousPlayGames2005,bergerTwoMoreClasses2007}), it turns out that degenerate game still plays an important role. For instance, there is some degenerate game in which no learning algorithm can converge to its NE (\cite{milionisImpossibilityTheoremGame2023}).

However, for $ 3\times 3 $ games without IIP, degeneracy does not affect the convergency of FP. And our method is still feasible for analysing degenerate games. To show this, we provide some examples in this section. 

\begin{example}\label{eg:degenerate}
	The first group of games $\big(A(k),B\big)$ have payoff matrices with parameter $k\in(-2,+\infty)$ as
	\begin{align}
		A = \begin{bmatrix}
			0     & \frac{1}{3}    & -1 \\[1.5ex]
			k+\frac{2}{3}    & 0     & -\frac{1}{3}     \\[1.5ex]
			k     & \frac{4}{3}    & 0
		\end{bmatrix},\quad B = \begin{bmatrix}
			0      & -\frac{4}{3}      & \frac{2}{3}      \\[1.5ex]
			-\frac{2}{3}  & 0      & -2      \\[1.5ex]
			\frac{1}{3}      & 0      & 0
		\end{bmatrix}.
	\end{align}
\end{example}

For all $ k\in(-2,+\infty) $, $ \bar{\mathbf{x}}\notin\Delta_{A}, \bar{\mathbf{y}}\notin\Delta_{B} $. And it is easy to check that when $ k\neq-\frac{2}{3} $, the game is nondegenerate and admits one uniqe NE. But when $k=-\frac{2}{3}$, we can see that once Player B chooses his first action, then Player A's first and second actions are both optimal. Hence $ \big(A(-\frac{2}{3}),B\big) $ is a degenerate game. Meanwhile, one can find that for any $x\in[0,1]$, $\big((x,1-x,0),e^{B}_{1}\big) $ are all NEs, and there is no other NE. Hence the set of NEs now is a continuum. 

For the degenerate game $ \big(A(-\frac{2}{3}),B\big) $, since $ \bar{\mathbf{x}}\notin\Delta_{A}, \bar{\mathbf{y}}\notin\Delta_{B} $ we can still conduct the same analysis with the projection mapping $ \varphi $. And we can see that the set of NEs is projected to the left edge of Cell \uppercase\expandafter{\romannumeral5}, see the second subgraph in Figure \ref{fig:burficationcontinuum}. We can still draw the vector field of PBRD in $ \tilde{\Delta} $, and we can see that the trajectory surrounds and approaches the set of NEs, but will never converge to any point in the set. Figure \ref{fig:cfpdegenerate} gives a simulation of CFP for game $ \big(A(-\frac{2}{3}),B\big) $.

\begin{figure}[htbp]
	\centering
	\includegraphics[width=1\linewidth]{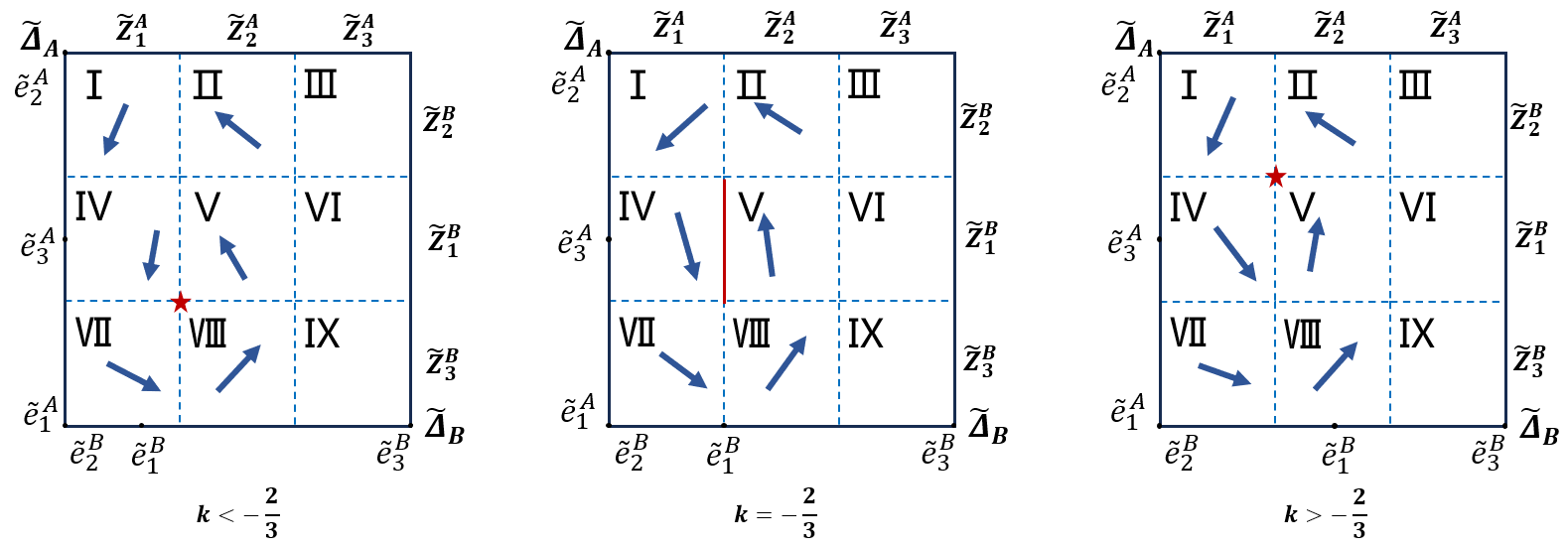}
	\caption{Regime shift of CFP dynamics in Example \ref{eg:degenerate}: when $ k<-\frac{2}{3} $, there exists an unique NE; as $ k $ increases to $ -\frac{2}{3} $, the NEs constitute a continuum and the vector field indicates the dynamics would approach to the continuum but does not converge to any point in it; when $ k>-\frac{2}{3} $, there exists an unique NE again.}
	\label{fig:burficationcontinuum}
\end{figure}

\begin{figure}[htbp]
	\centering
	\includegraphics[width=0.8\linewidth]{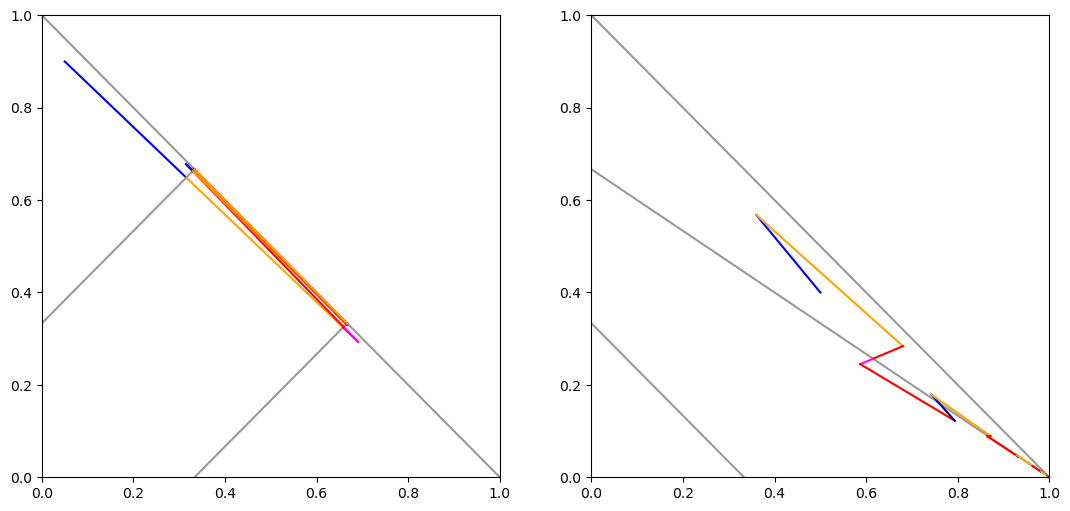}
	\caption{CFP in $\big(A(-\frac{2}{3}),B\big)$: the NE of this games is $\big((x,1-x,0),e^{B}_{1}\big), \forall x\in[0,1]$, and the trajectory approach the continuum of NE in $\Delta_{A}$, but does not approach any point of it.}
	\label{fig:cfpdegenerate}
\end{figure}

However, approaching to a continuum is not always the case for the degenerate game.
\begin{example}\label{eg:degenerate2}
	We consider another group of games $ \big(A(k),B\big) $ also with the parameter $ k\in(-2,+\infty) $. 
	\[ \begin{aligned}
		A = \begin{bmatrix}
			0     & \frac{2}{3}   & \frac{2}{3} \\[1.5ex]
			-\frac{2}{3}    & 0     & 1    \\[1.5ex]
			-\frac{2}{3}-k     & 1    & 0
		\end{bmatrix},\quad B = \begin{bmatrix}
			0      & 0      & -\frac{2}{3}      \\[1.5ex]
			\frac{1}{3}  & 0      & -\frac{2}{3}      \\[1.5ex]
			-1      & -\frac{1}{3}      & 0
		\end{bmatrix}.
	\end{aligned} \]
\end{example}
When $ k=-\frac{2}{3} $, the game is degenerate since Player A's first and third action are both optimal against Player B's first action. And this time, the points from set $\{\big((x,0,1-x),e^{B}_{1}\big)\mid x\in[0,1]\}$ are all NEs.  Figure \ref{fig:burfication} shows the PBRD under mapping $ \varphi $. But different from Example \ref{eg:degenerate}, the dynamics near the continuum would converge to only one endpoint $(e^{A}_{1},e^{B}_{1})$ of it.

	\begin{figure}[htbp]
		\centering
		\includegraphics[width=1\linewidth]{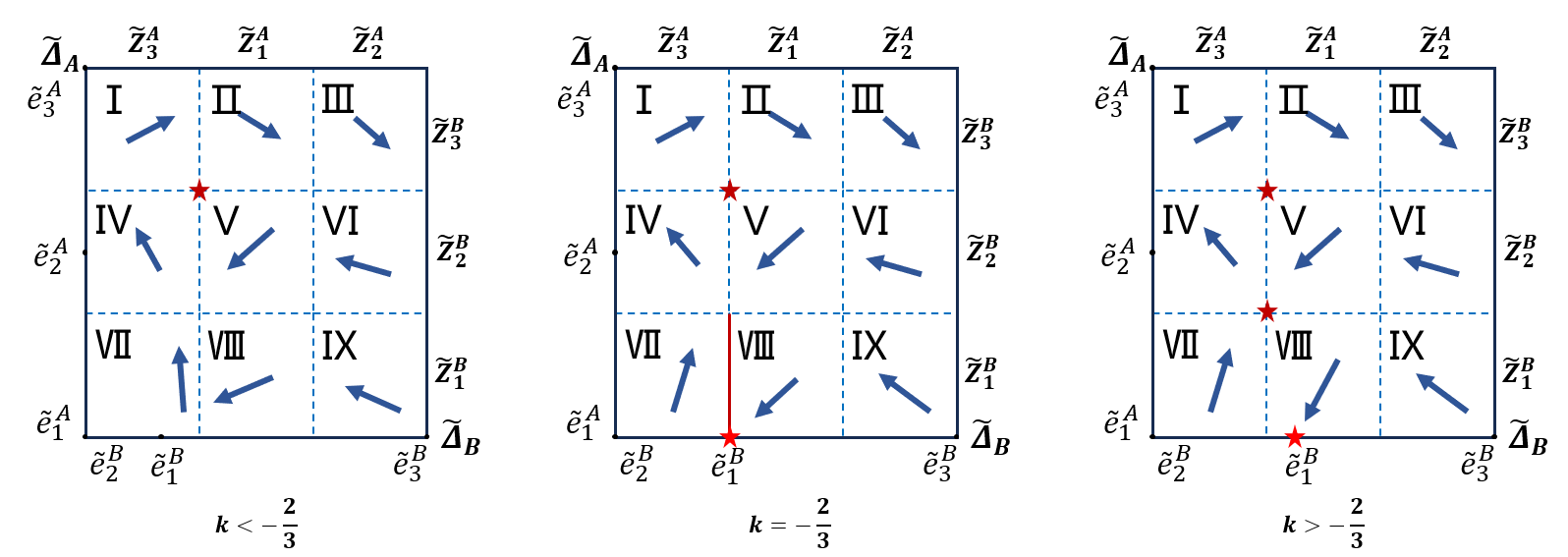}
		\caption{Regime shift of CFP dynamics in Example \ref{eg:degenerate2}. When $ k<-\frac{2}{3} $,  the game has an unique NE. As $ k $ increases to $ -\frac{2}{3} $, there are two components of NE, one is an isolated point, and the other is the edge of Cell \uppercase\expandafter{\romannumeral8}, i.e., a continuum. But the vector field in Cell \uppercase\expandafter{\romannumeral8} indicates that the dynamics will converge to only one endpoint of this edge. When $ k>-\frac{2}{3} $, the edge is no longer NE, and the game has three NEs. }
		\label{fig:burfication}
	\end{figure}

	Compare Example \ref{eg:degenerate2} and \ref{eg:degenerate}, we can see that although both games are degenerate and incorporate a continuum of NE when $k=-\frac{2}{3}$. But the dynamics nearby is completely different. Why? We note that in Example \ref{eg:degenerate}, the two endpoints of continuum are both sinks, one for $k<-\frac{2}{3}$ and the other for $k>-\frac{2}{3}$. We may view the continuum as the combination of sinks, hence it would attract the trajectory nearby. On the contrary, in Example \ref{eg:degenerate2} the lower endpoint is a sink but the upper endpoint is a saddle. As a result, all the trajectories would be attracted by the sink point, hence converge to one endpoint of the continuum. This may explain why the dynamics in these two examples are so different.
	
	On the other hand, we can view $ \big(A(k),B\big) $ as a path in the payoff matrix space $ (A,B)\in\mathbb{R}^{9}\times\mathbb{R}^{9} $ in both examples. Example \ref{eg:degenerate2} shows that the number of connected component of NE is not a homotopy invariant, which indicates the complexity of solving NEs. A small disturbance of $k$ near $-\frac{2}{3}$ in payoff matrix $A$ would even change the property of NEs dramatically.

\subsection{Higher-dimensional Games without IIP}
	In this section, we would continue our discussion about games with higher dimension. For simplicity, we suppose $ B=A^{T} $, and the initial value satisfies $ \mathbf{x}(t_{0})=\mathbf{y}(t_{0}) $. Hence $\operatorname{BR}_{A}=\operatorname{BR}_{B}$, and further the strategies $ \mathbf{x}(t) = \mathbf{y}(t), \forall t\geq t_{0} $. Thus the BRD could be simplified as 
	\[ \dot{\mathbf{x}}(t)=\operatorname{BR}_{A}\big(\mathbf{x}(t)\big)-\mathbf{x}(t). \]
	Now consider the following payoff matrix from \cite{hahnShapleyPolygonsGames2010}
	\begin{equation}\label{eg:4x4game}
		A=\begin{pmatrix}
			0 & -3 & \frac{3}{2} & 1 \\
			1 & 0 & -3 & -15 \\
			-15 & 1 & 0 & -1 \\
			-1 & -\frac{39}{20} & 1 & 0
		\end{pmatrix}. 
	\end{equation}
	The indifferent point of $ A $ is $ \bar{\mathbf{x}}=(-\frac{5}{2988}, \frac{185}{747}, \frac{1471}{1494}, -\frac{689}{2988})^{T} \notin \Delta_{A} $. Hence we can also define the projection mapping similar to $ 3\times3 $ games. To visualize the mapping, the mixed strategy space $\Delta_{A}$ can be represented as a three dimensional simplex $\{(x_{1},x_{2},x_{3})\mid 0\leq x_{i}\leq 1, \forall i\}$. In this example, $ \Delta_{A} $ is mapped onto one face $ \{(x_{1},x_{2},x_{3})\in\Delta_{A}\mid \sum_{i=1}^{3}x_{i}=1\} $. The Figure \ref{fig:4x4} shows the mapping in details.
	
	\begin{figure}[htbp]
		\centering
		\includegraphics[width=0.9\linewidth]{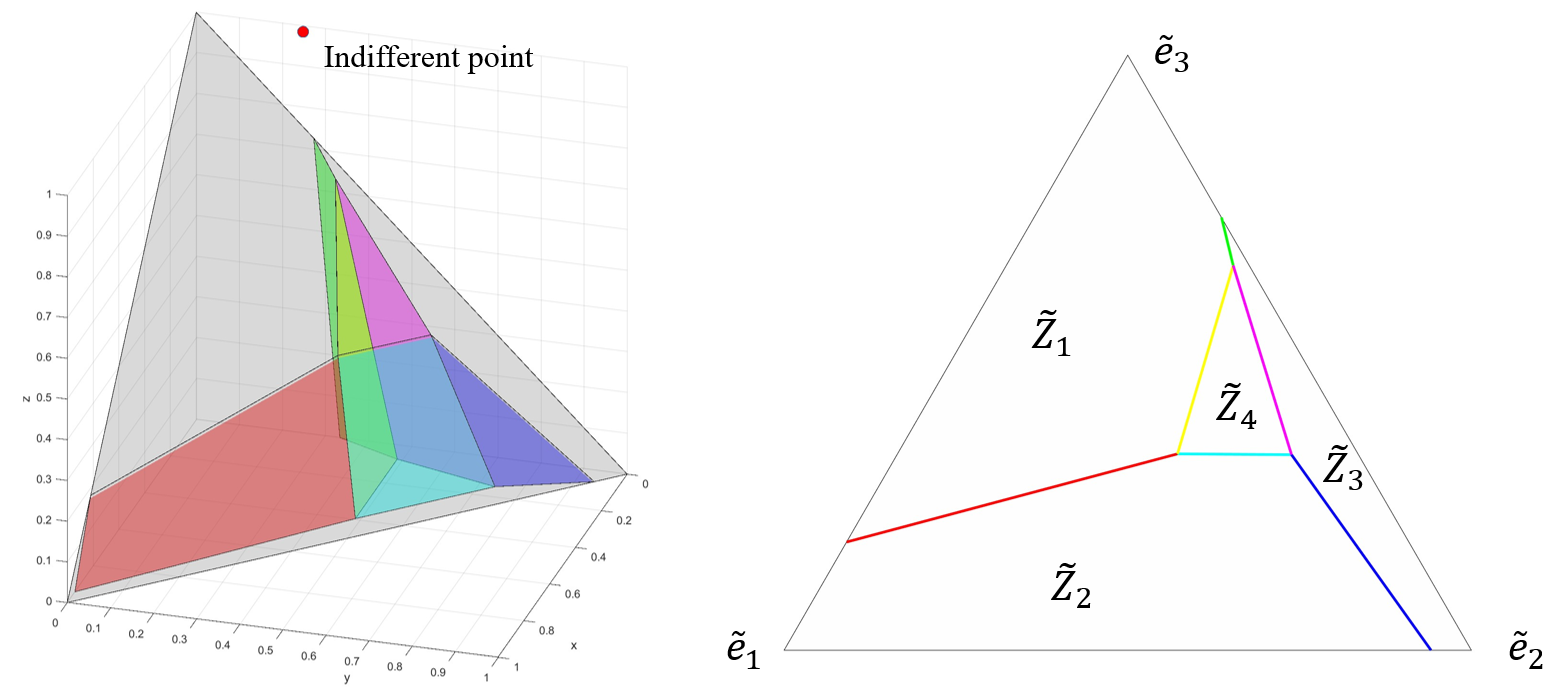}
		\caption{The original and projected strategy simplex space of Game (\ref{eg:4x4game}). In the left graph, the red point represents the indifferent point $\bar{\mathbf{x}}$. The gray plane is $ \{(x_{1},x_{2},x_{3})\in\Delta_{A}\mid \sum_{i=1}^{3}x_{i}=1\} $, and it is the face which we would project the whole simplex onto. The different planes, rather than indifferent lines, are colored with different colors. We preserve these colors in their the images under projection mapping in the right graph. And we also label each best response regions.}
		\label{fig:4x4}
	\end{figure}

	Under the projection mapping, we can understand the dynamics of CFP with much convenience. Specifically, We add the x-y axis to $\tilde{\Delta}$, and let $\tilde{e}_{1}$ to be the origin, and the side length is $\sqrt{2}$. The remaining vertex $ e^{4}=(0,0,0)^{T} $ is mapped to $ (-\frac{5}{3677}, \frac{740}{3677}, \frac{2942}{3677})^{T}\in\tilde{Z}^{A}_{1} $. Hence once the trajectory enters best response region for the forth action $ \tilde{Z}_{4}^{A} $, it would leave $ \tilde{Z}_{4}^{A} $ and enter $ \tilde{Z}^{A}_{1} $. Hence we can let the trajectory starts from $(0.8094,0.4041)$ in the red indifferent line, which is the left lower vertex of $\tilde{Z}_{4}$. The calculation shows that in $\tilde{\Delta}_{A}\textbackslash\tilde{Z}_4$, the trajectory would follow a pattern $\tilde{Z}_{1}\to\tilde{Z}_{2}\to\tilde{Z}_{3}$. At the next time that the trajectory returns to the red line, the coordinate of the point in the red line is $(0.3335,0.2773)$, which is farther from $\tilde{Z}_{4}$ than the start point $(0.8094,0.4041)$. Hence once the trajectory leaves $\tilde{Z}_{4}$, it would never return. And we can further infer that FP does not converge to NE in this case. So by the geometrical approach, we can prove the non-convergence result of game (\ref{eg:4x4game}) from \cite{hahnShapleyPolygonsGames2010} in a simpler way.  
	
	In addition, this example also indicates that being without IIP is not a sufficient condition for convergence of FP dynamics in games with higher dimension.

\section{Conclusions}\label{sec5:conclusion}
	Fictitious play is one of the most fundamental learning dynamics for computing NE. And finding games having FPP is one of the key focuses in studying FP. In this paper, we identify a new class of games having FPP, i.e., $3\times3$ games without IIP, which is fully based on the geometrical property of the game, including the location of NE and the partition of best response region. During the process, we devise a new projection mapping to reduce a high-dimensional dynamical system to a planar system. And to overcome the non-smoothness of the systems, we redefine the concepts of saddle and sink NE, which are proven to exist and help prove the convergence of CFP by separating the projected space into two parts. Furthermore, we show that our projection mapping can be extended to higher-dimensional and degenerate games.

	Apparently, there are a lot of future work. For all the $3\times 3$ games, based on the geometrical classification in this paper, the last piece to the complete characterization is still lacking, i.e. the class of games with $\bar{\mathbf{x}}\notin \Delta_A$ but $\bar{\mathbf{y}}\in\Delta_B$. For this game class, the projection mapping may not be well-defined in $\Delta_B$. However, this paper can still provide some inspiration. By similar analysis, we can infer the structure of NE for such games, then the concepts and results on the saddle and sink NE in this paper can be applicable. 
	Another future work is to combine our method with \cite{bergerFictitiousPlayGames2005} to study $3\times m$ games. Based on the method in \cite{bergerFictitiousPlayGames2005}, we can project the $m$-dimensional strategy simplex into a $3$-dimensional space. Then the analysis in this paper can be applied and it is possible to find some more $3\times m$ game classes having FPP. In both problems, geometrical approach is still helpful while more advanced techniques may be needed, especially concerning the higher-dimensional games.


\appendix
\renewcommand{\thesection}{\appendixname~\Alph{section}}
\section{Connection among FP, CFP and BRD}\label{app1}
For the sake of theoretical completeness, we provide here the derivation of CFP in (\ref{equ:cfp-def}) from FP in (\ref{equ:DFP}). However, the same calculation can also be found in the appendix of \cite{chen2022convergence}.

Suppose $t\in[1,+\infty)$ increase by $\delta$, then according to (\ref{equ:DFP}), we have
\[ \begin{aligned}
	\mathbf{x}(t+\delta) &= \frac{t}{t+\delta}\mathbf{x}(t)+\frac{\delta}{t+\delta}\operatorname{BR}_{A}(\mathbf{y}(t)),\\
	\mathbf{x}(t+\delta)-\mathbf{x}(t)&=\frac{\delta}{t+\delta}\big(\operatorname{BR}_{A}(\mathbf{y}(t))-\mathbf{x}(t)\big)
\end{aligned} \]
Let $\delta\to0$, we now have
\[ \dot{\mathbf{x}}=\lim_{\delta\to0}\frac{\mathbf{x}(t+\delta)-\mathbf{x}(t)}{\delta}=\frac{1}{t}\big(\operatorname{BR}_{A}(\mathbf{y}(t))-\mathbf{x}(t)\big). \]
Hence we get the equation of CFP.

\section{Proof of Lemma \ref{thm:existence-of-indiff-point}}\label{app:existence-of-indiff-point}
According to the definition, $ (\bar{\mathbf{x}},\bar{\mathbf{y}})\in\mathbb{R}^{2n} $ is the indifferent point if and only if we can find some values $ c_{1}, c_{2}\in\mathbb{R} $, such that 
$$
A\begin{pmatrix}
	y_{1} \\
	y_{2} \\
	\vdots \\
	y_{n}    
\end{pmatrix}=c_{1}\mathbf{1}_{n},\quad \begin{pmatrix}
	x_{1}, x_{2}, \cdots, x_{n}
\end{pmatrix}B=c_{2}\mathbf{1}_{n}^{T},
$$ 
where $ \mathbf{1}_{n} $ is a column vector whose elements are $ 1 $. 

We consider $\bar{\mathbf{y}}$ only for simplicity. Since $ y_{n}=\sum_{i=1}^{n-1}y_{i} $, we can substitute it to the above equation. Let $ A=(a_{ij}) $ and treat $ c_1 $ as another variable apart from $ y_{1},\cdots, y_{n-1} $, we can get
\begin{equation}\label{equ:existence-of-indiff-point}
	\begin{pmatrix}
		a_{11}-a_{1n} & a_{12}-a_{1n} & \cdots & a_{1(n-1)}-a_{1n} & -1\\
		a_{21}-a_{2n} & a_{22}-a_{2n} & \cdots & a_{2(n-1)}-a_{2n} & -1\\
		\vdots        & \vdots        &        & \vdots            & \vdots\\
		a_{n1}-a_{nn} & a_{n2}-a_{nn} & \cdots & a_{n(n-1)}-a_{nn} & -1
	\end{pmatrix}\begin{pmatrix}
		y_{1} \\
		\vdots \\
		y_{n-1}  \\
		c_1  
	\end{pmatrix}=\begin{pmatrix}
		-a_{1n} \\
		-a_{2n} \\
		\vdots \\
		-a_{n}    
	\end{pmatrix}
\end{equation}
We denote the matrix in (\ref{equ:existence-of-indiff-point}) as $ \tilde{A} $. Hence, the existence and uniqueness of $ \bar{y} $ is ensured once $ \det(\tilde{A})\neq0 $.

For $3\times 3$ game, if $ \det(\tilde{A})=0 $, then the row vectors of $\tilde{A}$ are linearly dependent, which means that, $\exists \lambda_1,\lambda_2\neq 0$
\[ \begin{aligned}
	a_{11}-a_{13}&=\lambda_{1}(a_{21}-a_{23})+\lambda_{2}(a_{31}-a_{32}),\\
	a_{12}-a_{13}&=\lambda_{1}(a_{22}-a_{23})+\lambda_{2}(a_{32}-a_{32}),\\
	-1&=-\lambda_{1}-\lambda_{2}.
\end{aligned} \]
On the other hand, we can calculate the indifferent lines $l^{A}_{12}$ by
\[ \begin{aligned}
	a_{13}-a_{23} + (a_{11}-a_{13}-a_{21}+a_{23})y_{1} + (a_{12}-a_{13}-a_{22}+a_{23})y_{2}=0.
\end{aligned} \]
But $$
\begin{aligned}
	a_{11}-a_{13}-a_{21}+a_{23}&=\lambda_{1}(a_{21}-a_{23})+\lambda_{2}(a_{31}-a_{32})-a_{21}+a_{23},\\
	&=\lambda_{2}(a_{31}-a_{32})-\lambda_{2}(a_{21}-a_{23}).
\end{aligned}
$$
Hence $l_{12}^{A}$ is $$a_{13}-a_{23} + \lambda_{2}(a_{31}-a_{32}-a_{21}-a_{23})y_{1} + \lambda_{2}(a_{32}-a_{33}-a_{22}+a_{23})y_{2}=0.$$ And $l^{A}_{23}$ is
\[ a_{23}-a_{33} + (a_{21}-a_{23}-a_{31}+a_{33})y_{1} + (a_{22}-a_{23}-a_{32}+a_{33})y_{2}=0. \]
We can see these two lines are parallel to each other since their slopes are both $-(a_{22}-a_{23}-a_{32}+a_{33})/(a_{21}-a_{23}-a_{31}+a_{33})$. Similarly, we can prove $l_{13}^{A}$ is parallel to $l_{23}^A$.

\section{Proportion of Matrix without IIP}\label{app:proportion}
Here we give the evidence that the class of games studied in this paper exists widely. We generate $100,000$ matrices $A$ whose elements are random numbers drawn from the interval $[0,1)$ with a uniform distribution, and we check whether the indifferent point $\bar{\mathbf{y}}$ belongs to $\Delta_B$ by solving equation (\ref{equ:existence-of-indiff-point}) in \ref{app:existence-of-indiff-point}. The calculation results show that for most payoff matrices $A$, $\bar{\mathbf{y}}\notin\Delta_B$. Even for a $3\times 3$ matrix, the proportion is bigger than $\frac{3}{4}$. Moreover, as the dimension of payoff matrix increases, the proportion increases as well, see Figure \ref{fig:proportion_IIP}.

\begin{figure}[htbp]
	\centering
	\includegraphics[width=0.5\linewidth]{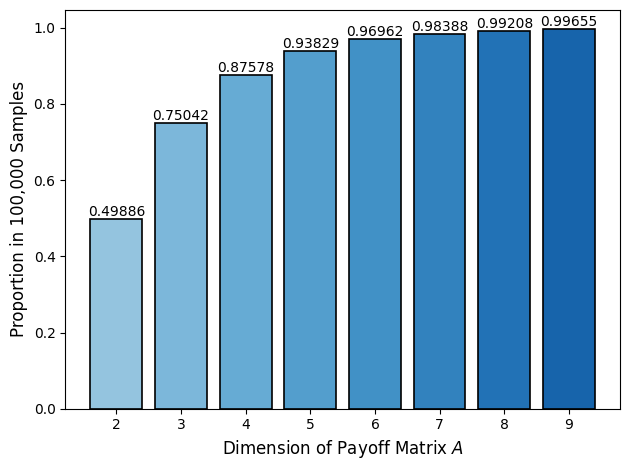}
	\caption{The Proportion of Matrix without IIP in $100,000$ Samples.}
	\label{fig:proportion_IIP}
\end{figure}

\section{Enumeration of different games}\label{app:enumeration}
In this section, we will distinguish different games according to the location of indifferent lines and the arrangement of best response regions. As indicated by Proposition \ref{pro:1}, for games without IIP, there are at most two indifferent lines in each strategy simplex. The location of them can be separated into the two cases in Figure \ref{fig:enum}.

\begin{figure}[ht]
	\centering
	\begin{subfigure}[c]{0.49\textwidth}
            \centering
		\includegraphics[width=0.55\linewidth]{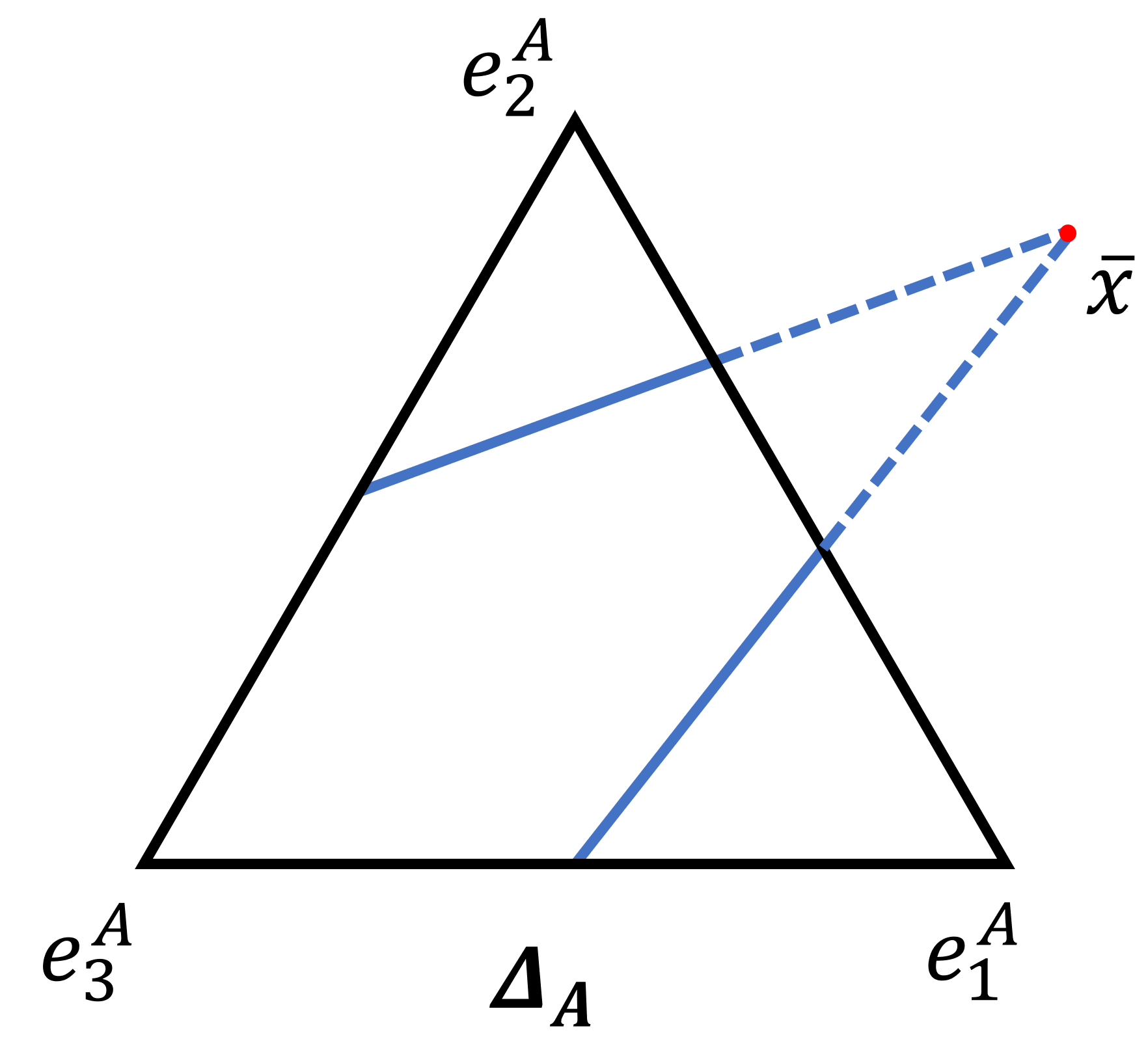}
		\caption{}
		\label{subfig:enum-a}
	\end{subfigure}
	\hfill
	\begin{subfigure}[c]{0.49\textwidth}
            \centering
		\includegraphics[width=0.55\linewidth]{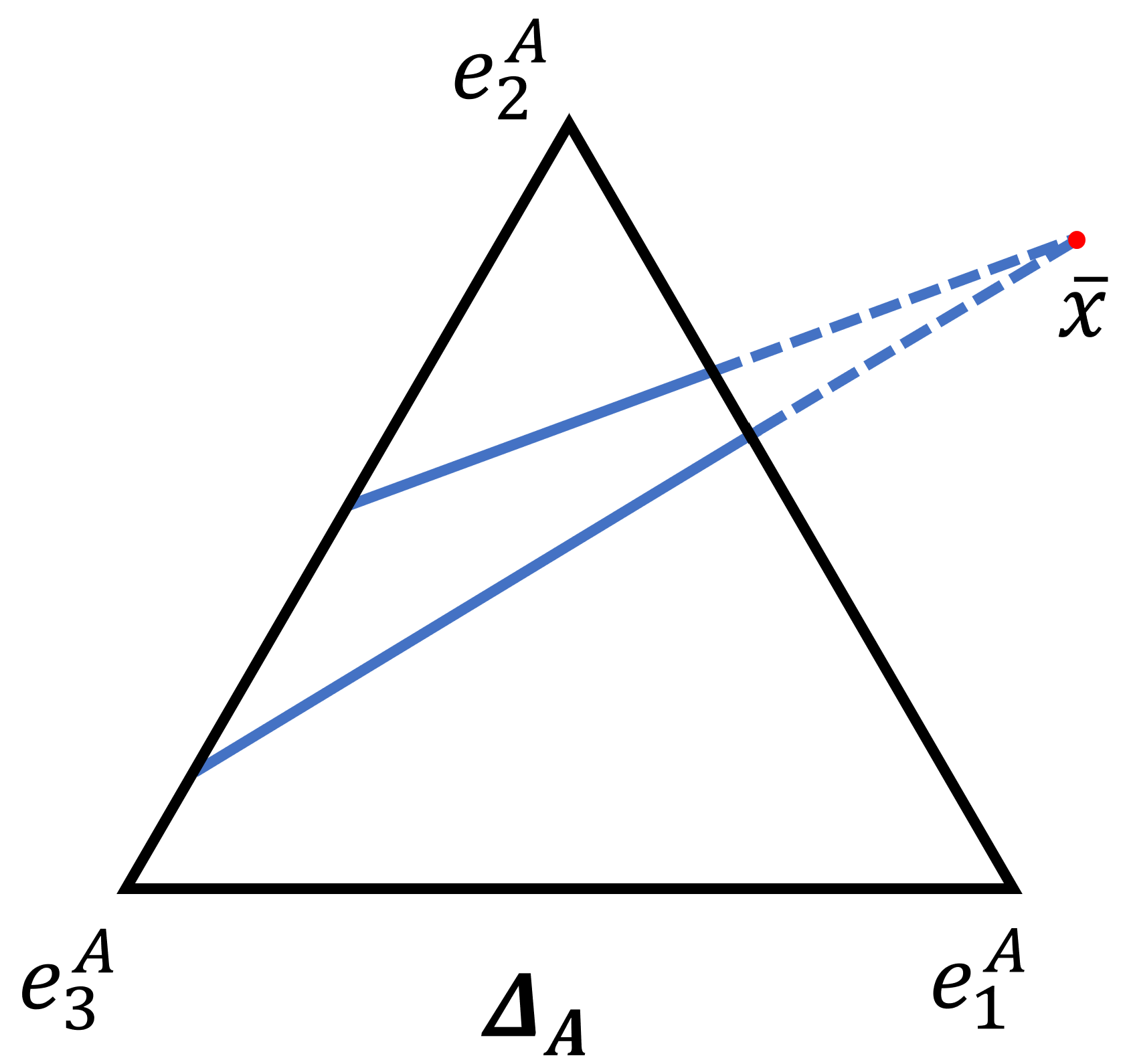}
		\caption{}
		\label{subfig:enum-b}
	\end{subfigure}
	\caption{Two cases for the location of indifferent lines\label{fig:enum}: in \cref{subfig:enum-a}, there is one edge of simplex having two intersection points with indifferent lines, and two edges having only one intersection points; in \cref{subfig:enum-b}, there are two edges of simplex having two intersection points with indifferent lines, but one edge having no intersection point.}
\end{figure}

The two indifferent lines would divide the strategy simplex into $3$ area, each corresponds to a best response regions. Figure \ref{fig:enumeration} shows the arrangement of best response regions, we only show the enumeration when the location of indifferent lines in each simplex $\Delta_{A}$ or $\Delta_{B}$ is demonstrated by Figure \ref{fig:enum}
\begin{figure}[htbp]
	\centering
	\includegraphics[width=1\linewidth]{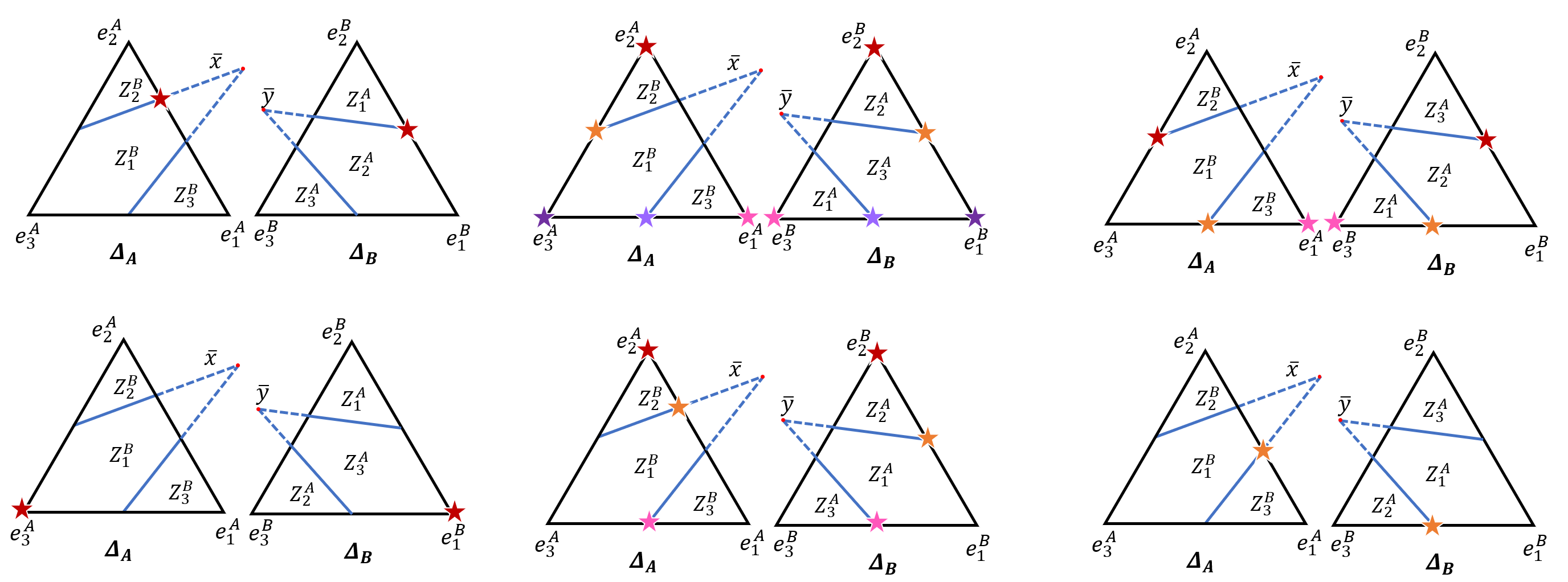}
	\caption{Enumeration of different arrangement of best response regions. Different colored strategy pairs indicate different NEs of the game.}
	\label{fig:enumeration}
\end{figure}

\section{The other case for Theorem \ref{thm:2-mixed}: Games with unique and pure NE}\label{app:PNE}
For games without IIP having an unique NE which is pure, we can prove that the only possible arrangement of best response regions is shown in Figure \ref{fig:proofpurene}. Here the only equilibrium is $ (e_{3}^{A},e_{1}^{B}) $, which belongs to $ Z_{1}^{B}\times Z_{3}^{A} $. If the trajectory starts from it, the best response is $ (e_{3}^{A},e_{1}^{B}) $, hence the trajectory would converge to NE directly. 

\begin{figure}[htbp]
	\centering
	\includegraphics[width=0.75\linewidth]{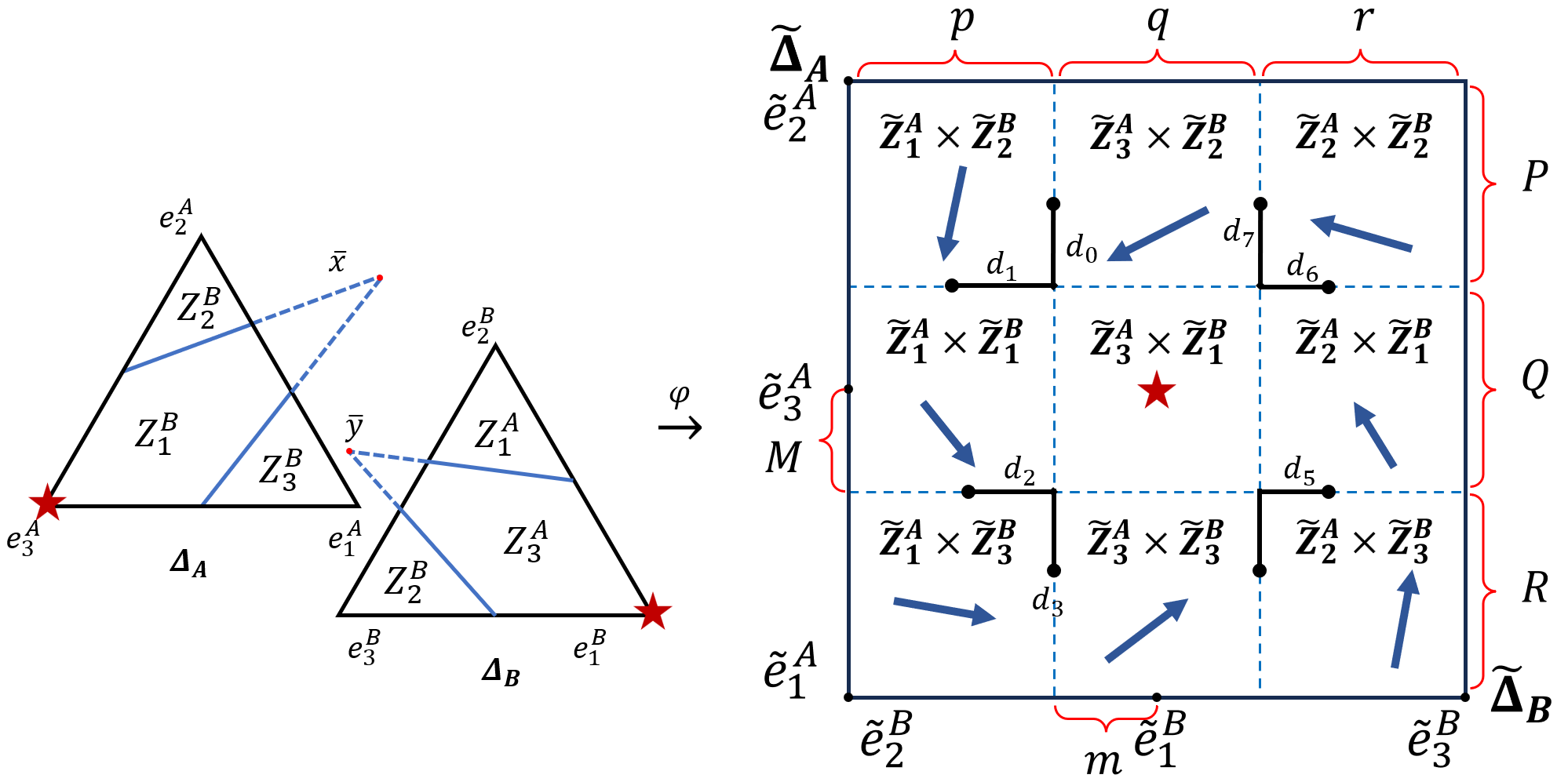}
	\caption{The game with unique and pure NE: the unique NE and its image under $ \phi $ are marked by red star in $ \Delta $ and $ \tilde{\Delta} $; the blue arrows show the vector field in each cell.}
	\label{fig:proofpurene}
\end{figure}

As a result, in order to prove the convergence of PBRD, we only need to prove the trajectory will enter $ \tilde{Z}^{B}_{3}\times\tilde{Z}_{1}^{A} $, or Cell \uppercase\expandafter{\romannumeral5}. First we can draw the vector field in each cell, from which we can see that if the PBRD trajectory never enters Cell \uppercase\expandafter{\romannumeral5}, it must go through all the other eight cells and surround Cell \uppercase\expandafter{\romannumeral5} endlessly. Next We will show that this is impossible.

To capture the distance of PBRD trajectory to Cell \uppercase\expandafter{\romannumeral5}, we define $ d $ as
\[ d(x,y) = \inf_{z}\{\Vert (x,y)-z \Vert_{2}\mid z\in \tilde{Z}^{B}_{3}\times\tilde{Z}_{1}^{A}\},\quad \forall (x,y)\notin  \tilde{Z}^{B}_{3}\times\tilde{Z}_{1}^{A}. \]
We mainly focus on the intersection points of the trajectory and the edges of the cells. Without loss of generality, let the trajectory start from the right edge of Cell \uppercase\expandafter{\romannumeral1}, i.e., $ \{(a,y)\mid Q+R<y\leq 1\} $. We denote the distance between the initial point and Cell \uppercase\expandafter{\romannumeral5} as $ d_{0} $, which is just $ y-(B+D) $.

Starting from that point, since the trajectory will never enter Cell \uppercase\expandafter{\romannumeral5}, it would follow a path of cells \uppercase\expandafter{\romannumeral1} $ \to $ \uppercase\expandafter{\romannumeral4} $ \to $ \uppercase\expandafter{\romannumeral7} $ \to $ \uppercase\expandafter{\romannumeral8} $ \to $ \uppercase\expandafter{\romannumeral9} $ \to $ \uppercase\expandafter{\romannumeral6} $ \to $ \uppercase\expandafter{\romannumeral3} $ \to $ \uppercase\expandafter{\romannumeral2} and return to Cell \uppercase\expandafter{\romannumeral1}. 
During this process, we can get a sequence of distance $ \{d_{n}\}, 0\leq n\leq8 $. 
$ d_{8} $ is the second time that the trajectory enters Cell \uppercase\expandafter{\romannumeral1}. 

Now we can define a Poincar\'e mapping $ f:d_{0}\mapsto d_{8} $, and it is the composition of $ 8 $ M\"obius transformation, denoted by $ \phi_{i}:d_{i}\mapsto d_{i+1}, i=0,\cdots, 7 $, and $ f=\phi_{7}\circ \phi_{6}\circ\cdots \phi_{0} $.

First, we point out that the four transformations $ \phi_{1},\phi_{3},\phi_{5},\phi_{7} $ satisfy $ d_{i+1}=\phi_{i}(d_{i})<d_{i} $, which correspond to the dynamics in Cell \uppercase\expandafter{\romannumeral4}, \uppercase\expandafter{\romannumeral8}, \uppercase\expandafter{\romannumeral6} and \uppercase\expandafter{\romannumeral2} respectively. For example, when the trajectory enters Cell \uppercase\expandafter{\romannumeral4}, $ \phi_{1} $ can be solved via similar triangle as 
\begin{equation}\label{equ:app-pne1}
	\frac{d_{2}+m}{d_{1}+m}=\frac{R}{Q+R}\Rightarrow d_{2}=\phi_{1}(d_{1})=\frac{R}{Q+R}d_{1}-\frac{Q}{Q+R}m<d_{1}.
\end{equation}
Similarly, 
\[ 	\begin{aligned}
	&\frac{d_{4}+M}{d_{3}+M}=\frac{r}{q+r}\Rightarrow d_{4}=\phi_{3}(d_{3})=\frac{r}{r+q}d_{3}-\frac{q}{r+q}M<d_{3},\\
	&\frac{d_{6}+q-m}{d_{5}+q-m}=\frac{P}{P+Q}\Rightarrow d_{6}=\phi_{5}(d_{5})=\frac{P}{P+Q}d_{5}-\frac{Q}{P+Q}(q-m)<d_{5},\\
	&\frac{d_{8}+Q-M}{d_{7}+Q-M}=\frac{p}{p+q}\Rightarrow d_{8}=\phi_{7}(d_{7})=\frac{p}{p+q}d_{7}-\frac{q}{p+q}(Q-M)<d_{7}.\label{equ:app-pne2}
\end{aligned} \]
Hence we have
\[ f(d_{0})=\phi_{7}\circ \phi_{6}\circ\phi_{5}\circ\phi_{4}\circ\phi_{3}\circ\phi_{2}\circ\phi_{1}\circ\phi_{0}(d_{0})<\phi_{6}\circ\phi_{4}\circ\phi_{2}\circ\phi_{0}(d_{0}). \]
For the other four M\"obius transformations, we have
\[ \begin{aligned}
	&\frac{d_{1}}{d_{0}}=\frac{p}{Q+R+d_{0}}\Rightarrow d_{1}=\phi_{0}(d_{0})=\frac{p\cdot d_{0}}{Q+R+d_{0}},\\
	&\frac{d_{3}}{d_{2}}=\frac{R}{q+r+d_{2}}\Rightarrow d_{3}=\phi_{2}(d_{2})=\frac{R\cdot d_{2}}{q+r+d_{2}},\\
	&\frac{d_{5}}{d_{4}}=\frac{r}{P+Q+d_{4}}\Rightarrow d_{5}=\phi_{4}(d_{4})=\frac{r\cdot d_{4}}{P+Q+d_{2}},\\
	&\frac{d_{7}}{d_{6}}=\frac{P}{p+q+d_{6}}\Rightarrow d_{7}=\phi_{6}(d_{6})=\frac{P\cdot d_{6}}{p+q+d_{6}}.
\end{aligned} \]
Together we have 
\[ \begin{aligned}
	f(d_{0})&<\phi_{6}\circ\phi_{4}\circ\phi_{2}\circ\phi_{0}(d_{0})\\
	&= \frac{p P r Rd_{0}}{(p + q) (P + Q) (q + r) (Q + R) + \ast}<d_{0}.
\end{aligned} \]
Here $ \ast $ represents a long item with respect to $ d_{0} $, and the last inequality holds because $ (p+q)(P+Q)(q+r)(Q+R)>p\cdot Q\cdot r\cdot R $.
Hence $ f(d_{0}) $ is contractive, and the trajectory would approach Cell \uppercase\expandafter{\romannumeral5}, i.e., $ f^{n}(d_{0})\to 0 $. 

When $n$ is large enough, $f^{n}(d_{0})$ is sufficiently small. Then according to the above equations of $\phi_{i},\ i=1,3,5,7$,  
there exists one $\phi_i$ 
such that $\phi_i(d_i)<0$, meaning the trajectory will no longer surround Cell \uppercase\expandafter{\romannumeral5} but enter it. And we get the convergent result.

\section{The other two cases in the proof of Proposition \ref{pro:exist-saddle}} \label{app:exist-saddle}
\begin{figure}[htbp]
	\centering
	\includegraphics[width=1\linewidth]{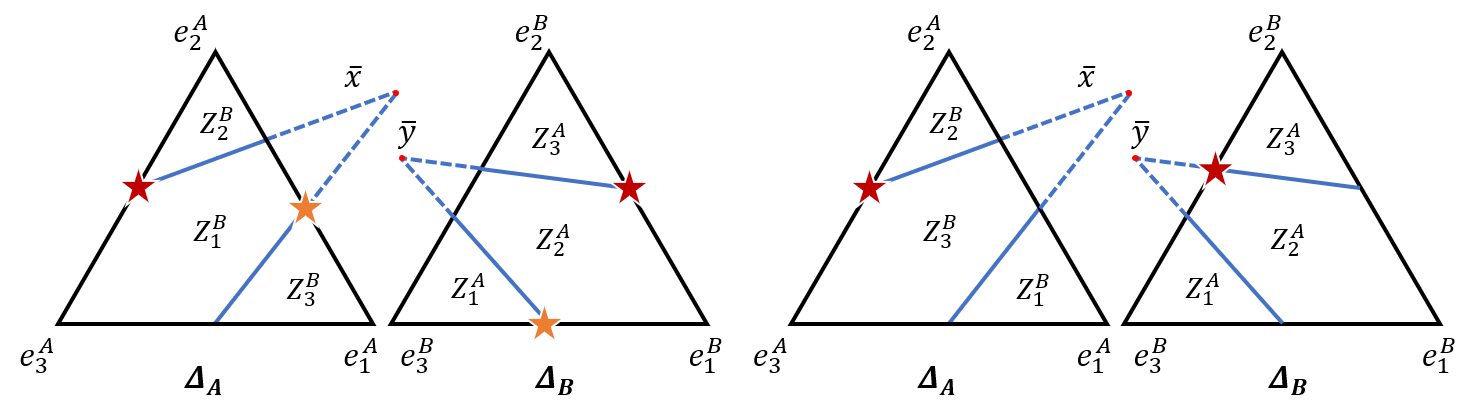}
	\caption{Different arrangements of best response regions when the game has a sink NE.}
	\label{fig:appendixsaddle}
\end{figure}

We now complete our proof of existence of saddle NE in Proposition \ref{pro:exist-saddle}. There are two cases remained,
\begin{enumerate}
	\item the game has a pure NE which lies on the vertex of $ \tilde{\Delta} $.
	\item the game has a mixed NE $(\mathbf{x},\mathbf{y})$ which is a sink.
\end{enumerate}
The proofs for the two cases are totally the same, and we only give the proof for the second case.

We suppose the upper left vertex of Cell \uppercase\expandafter{\romannumeral5} is the image of sink NE under $\varphi$. Then there are four cases, depending on $\mathbf{x}\in e^{A}_{2}e^{A}_{3}$ or $e^{A}_{1}e^{A}_{2}$ and $\mathbf{y}\in e^{B}_{1}e^{B}_{2}$ or $e^{B}_{2}e^{B}_{3}$. We note that when the location of $(\mathbf{x},\mathbf{y})$ is determined, the arrangement of best response regions is fixed to make sure $(\mathbf{x},\mathbf{y})$ is a sink NE. We can utilize this arrangement and get contradiction. Appendix \ref{fig:appendixsaddle} shows all these cases.

\paragraph{Case 1: $\mathbf{x}\in e^{A}_{2}e^{A}_{3}, \mathbf{y}\in e^{B}_{1}e^{B}_{2}$ or $\mathbf{x}\in e^{A}_{1}e^{A}_{2}, \mathbf{y}\in e^{B}_{2}e^{B}_{3}$.} Since $(\mathbf{x},\mathbf{y})$ is a NE, $\mathbf{x}= e^{A}_{2}e^{A}_{3}\cap l^{B}_{12}$, $\mathbf{y}=e^{B}_{1}e^{B}_{2}\cap l^{A}_{23}$ (see the red points in the left subgraph in Figure \ref{fig:appendixsaddle}). And since $(\mathbf{x},\mathbf{y})$ is a sink, Cell \uppercase\expandafter{\romannumeral1} is $\tilde{Z}_{2}^{B}\times\tilde{Z}_{3}^{A}$, Cell \uppercase\expandafter{\romannumeral2} is $\tilde{Z}_{2}^{B}\times\tilde{Z}_{2}^{A}$, Cell \uppercase\expandafter{\romannumeral4} is $\tilde{Z}_{1}^{B}\times\tilde{Z}_{3}^{A}$, Cell \uppercase\expandafter{\romannumeral5} is $\tilde{Z}_{1}^{B}\times\tilde{Z}_{2}^{A}$. Then we can find that the profile $(e^{A}_{1}e^{A}_{2}\cap l^{B}_{13}, e^{B}_{1}e^{B}_{3}\cap l^{A}_{12})$ is a saddle NE (see the yellow points in the left subgraph in Figure \ref{fig:appendixsaddle}). The case for $\mathbf{x}\in e^{A}_{1}e^{A}_{2}, \mathbf{y}\in e^{B}_{2}e^{B}_{3}$ is the same.

\paragraph{Case 2: $\mathbf{x}\in e^{A}_{2}e^{A}_{3}, \mathbf{y}\in e^{B}_{2}e^{B}_{3}$ or $\mathbf{x}\in e^{A}_{1}e^{A}_{2}, \mathbf{y}\in e^{B}_{1}e^{B}_{2}$.} When $\mathbf{x}\in e^{A}_{2}e^{A}_{3}, \mathbf{y}\in e^{B}_{2}e^{B}_{3}$ (see the red points in the middle subgraph in Figure \ref{fig:appendixsaddle}), since $(\mathbf{x},\mathbf{y})$ is a sink NE, the arrangement of best response should satisfy $e^{A}_{2}\in Z^{B}_{2}, e^{A}_{3}\in Z^{B}_{3}$ and $e^{B}_{2}\in Z^{A}_{3}, e^{B}_{1}\in Z^{A}_{2}$. Then we can check there is no other NE, which contradicts to the assumption of multiple NEs in Proposition \ref{pro:exist-saddle}. The case for $\mathbf{x}\in e^{A}_{1}e^{A}_{2}, \mathbf{y}\in e^{B}_{1}e^{B}_{2}$ is the same.

\bibliographystyle{unsrt}  
\bibliography{references}

\end{document}